\documentclass[11pt,reqno]{amsart}

\DeclareMathAlphabet\mathbfcal{OMS}{cmsy}{b}{n}
\usepackage{xcolor}
\usepackage{amssymb}
\usepackage{amscd}
\usepackage{amsfonts}
\usepackage{mathrsfs}
\usepackage{setspace}
\usepackage{version}
\usepackage{mathtools}
\usepackage{textcomp}
\usepackage{stmaryrd}
\usepackage[english]{babel}
\usepackage[colorlinks=true,linkcolor=blue]{hyperref}

\newtheorem{theorem}{Theorem}[section]
\newtheorem{lemma}[theorem]{Lemma}
\newtheorem{proposition}[theorem]{Proposition}
\newtheorem{corollary}[theorem]{Corollary}
\theoremstyle{definition}
\newtheorem{definition}[theorem]{Definition}
\newtheorem{assumption}[theorem]{Assumption}
\newtheorem{example}[theorem]{Example}

\theoremstyle{remark}
\newtheorem{remark}[theorem]{Remark}
\numberwithin{equation}{section}

\newcommand{\I}{I_{\rm min}}

\newcommand\R{\mathbb{R}}
\newcommand\N{\mathbb{N}}
\renewcommand\P{\mathbb{P}}

\newcommand{\vertiii}[1]{{\left\vert\kern-0.25ex\left\vert\kern-0.25ex\left\vert #1 
    \right\vert\kern-0.25ex\right\vert\kern-0.25ex\right\vert}}

\begin{document}
	

\title[]{Weakly maxitive set functions and their \\ possibility distributions}
	
    \author{Michael Kupper}
	\address{Department of Mathematics and Statistics, University of Konstanz}
	\email{kupper@uni-konstanz.de}
	
		\author{Jos\'e M.~Zapata}
	\address{Departamento de estadística e investigación operativa, Universidad de Murcia}
	\email{jmzg1@um.es}
	
	\date{\today}
	
	\begin{abstract}
	The Shilkret integral with respect to a completely maxitive capacity is fully determined by a possibility distribution. 
	In this paper, we introduce a weaker topological form of maxitivity and show that under this assumption the Shilkret integral is still determined by its possibility distribution for functions that are sufficiently regular.  Motivated by large deviations theory, we provide a Laplace principle for maxitive integrals and characterize the possibility distribution under certain separation and convexity assumptions.  Moreover, we show a maxitive integral representation result for weakly maxitive non-linear expectations. The theoretical results are illustrated by providing large deviations bounds for sequences of capacities, and by deriving a monotone analogue of Cram\'{e}r's theorem.   
	
		
	\end{abstract}

		\thanks{We thank two anonymous referees for their helpful suggestions and remarks.}
	
	\maketitle
	
	\setcounter{tocdepth}{1}
	
\section{Introduction}
Non-additive set functions,  fuzzy measures and capacities play an important role in the theory of decision making under risk and uncertainty.
They appear in possibility theory, idempotent/tropical mathematics and related fields, where maxitive measures  and their maxitive integrals are studied; see, e.g.,~\cite{dubois2006,dubois2015possibility,maslov,shilkret1971maxitive,zadeh1978fuzzy}. 
In this article we focus on the \emph{Shilkret integral} introduced in~\cite{shilkret1971maxitive}. 
An extension of this integral to functions taking negative values, and its complete characterization  was given in~\cite{cattaneo2014maxitive,cattaneo2016maxitive}. On the one hand, if $\Pi$ is a possibility measure \cite{zadeh1978fuzzy}, then the Shilkret integral  is given by
\[
\int^S f\, {\rm d}\Pi=\sup_{x} f(x) \pi(x),
\]
where $\pi$ denotes the corresponding possibility distribution. In this case, both the set function $\Pi$ and the Shilkret integral $\int^S f\, {\rm d}\Pi$ are completely maxitive.
On the other hand, for the particular set function $\bar{\Pi}_A:=\limsup_{n\to\infty} \mathbb{P}(X_n\in A)^{1/n}$, where $(X_n)_{n\in\N}$ is a sequence of random variables,
the Laplace principle from the theory of large deviations ensures under reasonable assumptions that 
\[\int^S f\, {\rm d}\bar\Pi=\sup_{x} f(x) e^{-I(x)},\] where $I$ denotes the 
rate function; cf.~Puhalskii~\cite{puhalskii2001large}. Although the two representations for the Shilkret integral are very similar, in general, the second one is only valid for certain continuous functions. Thus, the set function $\bar\Pi$ satisfies only a weaker form of maxitivity, and consequently is not a possibility measure. Nevertheless, $\bar\pi(x)=e^{-I(x)}$ can still be viewed as a corresponding possibility distribution. In this article, we characterize the class of set functions which admit a possibility distribution in a weaker topological sense by means of different notions of maxitivity. 

In Section~\ref{sec:possdistribution}, we introduce some basics on possibility distributions and motivate our results.
Section~\ref{sec:maxInt} includes the key bounds on maxitive integrals and their connection to weakly maxitive set functions.
In Section~\ref{sec:integralRep}, we focus on a maxitive integral representation result for weakly maxitive non-linear expectations.
The results are then applied in Section \ref{sec:monLDP}, where we show that the basic results of large deviations theory are valid for general maxitive integrals; in particular, the equivalence 
between the monotone large deviation principle and the monotone Laplace principle. In Section \ref{sec:convRateFunc}, we provide conditions ensuring that the rate function is convex and study the corresponding representation.  
Finally, in Section \ref{sec:examples},  the theoretical results are illustrated with two examples. 
On the one hand, we study the asymptotic behaviour of a sequence of capacities by providing some large deviations bounds, on the other hand we establish a monotone analogue of Cram\'{e}r's theorem for the sample mean of i.i.d.~sequences.  The paper concludes with an appendix, where we provide some separation results in preordered topological groups.

\section{Background and motivation}\label{sec:possdistribution}
Let $E$ be a non-empty set and $\mathcal{A}$ be a collection of subsets of $E$ such that $\emptyset,E\in\mathcal{A}$. A set function $\Pi\colon \mathcal{A}\to \R$ is called a  \emph{capacity}, if $\Pi_\emptyset=0$, $\Pi_E=1$, and $\Pi_A\le\Pi_B$ whenever $A\subset B$.\footnote{Typically $\mathcal{A}$ is endowed with some algebraic structure. Here, for the sake of generality we do not assume any structrure on $\mathcal{A}$.}
Among the most important examples of capacities, we find probability measures in standard probability theory~\cite{kolmogorov}, possibility measures in possibility  
theory~\cite{dubois2006}, and upper/lower probabilities in the theory of imprecise probabilities~\cite{walley1991statistical}. A capacity $\Pi$ is called a \emph{possibility measure} \cite{zadeh1978fuzzy},  if there exists a \emph{possibility distribution}  $\pi\colon E\to [0,1]$ such that 
\begin{equation}\label{possdistribution}
	\Pi_A=\sup_{x\in A} \pi(x)\quad\mbox{ for all }A\in \mathcal{A}.
\end{equation}
A capacity $\Pi$ which admits a possibility distribution is automatically completely maxitive, i.e., 
\begin{equation}
	\label{eq:possibilityMeasure}
	\Pi_{A}\le \vee_{i\in {\mathcal I}}\Pi_{A_i}\quad\mbox{ whenever }A\subset \cup_{i\in {\mathcal I}} A_i,
\end{equation}
for all $A\in\mathcal{A}$ and every family $(A_i)_{i\in {\mathcal I}}\subset \mathcal{A}$. Conversely, if $\mathcal{A}$ is closed under arbitrary unions and complements, then any supremum preserving capacity $\Pi$ admits a 
possibility distribution which is uniquely determined by $\pi(x)=\Pi_{\cap\{A\in\mathcal{A}\colon x\in A\}}$ for all $x\in E$; cf.~\cite{cooman1999}.
Under additional continuity on the capacity and the assumption that $\mathcal{A}$ is rich enough,\footnote{E.g.~the collection $\mathcal{A}$ forms a $\sigma$-algebra on a separable metric space which contains all Borel sets.} the existence of a possibility distribution is guaranteed if $\Pi$ is only finitely maxitive, i.e., $\Pi_A\le \vee_{i=1}^n \Pi_{B_i}$ for all $A,B_1,\dots,B_n\in\mathcal{A}$, $n\in\N$,  with $A\subset \cup_{i=1}^n B_i$; cf.~\cite{arslanov2004existence,murofushi1993}.

The \emph{Shilkret integral} of a function $f\colon E\to [0,\infty]$ with respect to a capacity $\Pi$ is defined by
\begin{equation}
	\label{eq:Shilkret}
	\int^S f \,{\rm d}\Pi:=\underset{c\in (0,\infty)}\sup c \Pi_{\{f > c\}}.
\end{equation} 
Originally, the Shilkret integral was introduced for supremum preserving capacities, also called maxitive  probabilities or idempotent probabilities \cite{shilkret1971maxitive}.  
However, this definition is valid for general capacities (not necessarily maxitive) provided that $f$ is  $\mathcal{A}$-measurable, i.e., $\{f> c\}\in\mathcal{A}$ for all $c\in \mathbb{R}$. 
In case that $\Pi$ admits a possibility distribution $\pi$, then the Shilkret integral takes the form
\begin{equation}
	\label{eq:possibilityIntegral}
	\int^S f\, {\rm d}\Pi=\sup_{x\in E} f(x) \pi(x)
\end{equation}
for all $\mathcal{A}$-measurable functions $f\colon E\to [0,\infty]$. 

In this article, we analyze to what extent a capacity is still determined by 
a possibility distribution if the representation  \eqref{eq:possibilityIntegral} only holds for a certain class of $\mathcal{A}$-measurable functions, e.g., all continuous, continuous bounded, or continuous increasing functions.\footnote{Here, we  assume that $E$ is a topological preordered space.} 
We will see that this exactly holds  when the capacity satisfies a weaker form of maxitivity.\footnote{I.e.~$\Pi_A\le \vee_{i=1}^n \Pi_{B_i}$ 
for (upwards closed) closed sets $A$ and (upwards closed) open sets $B_1,\dots,B_n$, $n\in\N$, with $A\subset \cup_{i=1}^n B_i$. 
We emphasize that no additional continuity on $\Pi$ is required.} To do so, we will provide bounds for the Shilkret integral with respect to general capacities.  
For a capacity $\Pi$ and a function $\pi\colon E\to [0,1]$, a key result of this work states that 
\begin{equation}\label{eq:lowerShilkret}
	\int^S f\,{\rm d} \Pi\ge \sup_{x\in E} f(x)\pi(x)\quad\mbox{if and only if}\quad \Pi_O\ge \sup_{x\in O} \pi(x),
\end{equation}
where the first inequality holds for all (increasing) lower semicontinuous functions $f$, and the second inequality holds for all (upwards closed) open sets $O$. Likewise,  
\begin{equation}\label{eq:upperShilkret}
	\int^S f\,{\rm d} \Pi\le \sup_{x\in E} f(x)\pi(x)\quad\mbox{if and only if}\quad \Pi_C\le \sup_{x\in C} \pi(x),
\end{equation}
where the first inequality holds for all (increasing) upper semicontinuous functions $f$, and the second inequality holds for all (upwards closed) closed sets $C$. 
Instead of assuming the rather restrictive assumption \eqref{possdistribution}, the relation between the capacity and its possibility distribution is relaxed to 
inequalities which are required only for certain nice topological sets. As a consequence, if the capacity $\Pi$ satisfies both the upper and lower bound, then the Shilkret integral has the representation \eqref{eq:possibilityIntegral} for all (increasing) continuous functions. 

The Shilkret integral is only defined for non-negative functions. 
In order to deal with real-valued functions, we present and prove our results in terms of the maxitive integral introduced by Cattaneo~\cite{cattaneo2016maxitive}, which is obtained as a transformation of the Shilkret integral. We say that a set function $J\colon \mathcal{A}\to [-\infty,0]$ is a \emph{concentration} if $J_E=0$, $J_{\emptyset}=-\infty$, and $J_A\le J_B$ whenever $A\subset B$. In other words,  $J$ is a concentration if and only if $e^J$ is a capacity.   
The maxitive integral of an $\mathcal{A}$-measurable function $f\colon E\to[-\infty,\infty)$ with respect to the concentration $J$ is defined as
\begin{equation}\label{def:convint}
	\phi_J(f):=\log \int^S e^f \,{\rm d} e^J=\underset{c\in \R}\sup\{c+ J_{\{f>c\}}\}.
\end{equation}
As discussed in \cite{cattaneo2016maxitive}, the functional $\phi_J$ shares the properties of a monetary risk measure~\cite{follmer2016stochastic}, 
and satisfies in particular the \emph{translation property} $\phi_J(f+c)=\phi_J(f)+c$ for all $c\in\mathbb{R}$.
In contrast, the Shilkret integral~\eqref{eq:Shilkret} fails the translation property unless the capacity $\Pi$ only assumes the values $0$ and $1$; cf.~\cite{de2001integration}. 
In particular, aside from this degenerate case, the Shilkret integral is neither a coherent prevision, nor a monetary risk measure.  
By defining the rate function $I\colon E\to [0,\infty]$ by $I(x):=-\log\pi(x)$, the bound \eqref{eq:lowerShilkret} takes the form
\begin{equation}\label{bound1}
	\phi_J(f)\ge \sup_{x\in E}\{f(x)-I(x)\}\quad\mbox{if and only if}\quad J_O\ge -\inf_{x\in O} I(x),
\end{equation}
for all (increasing) lower semicontinuous functions $f$, and all (upwards closed) open sets $O$, and  
the bound \eqref{eq:upperShilkret} translates to 
\begin{equation}\label{bound2}
	\phi_J(f)\le \sup_{x\in E}\{f(x)-I(x)\}\quad\mbox{if and only if}\quad J_C\le -\inf_{x\in C} I(x),
\end{equation}
for all (increasing) upper semicontinuous functions $f$, and all (upwards closed) closed sets $C$.

    Similar type of bounds appear in the theory of large deviations, where the capacity has the special form ${\Pi}_A:=\limsup_{n\to\infty} \mathbb{P}(X_n\in A)^{1/n}$
	for a sequence $(X_n)_{n\in\N}$ of random variables with values in a completely regular topological space $E$. 
	In that case, the corresponding concentration is given by  $J_A=\limsup_{n\to\infty}\tfrac{1}{n}\log \mathbb{P}(X_n\in A)$ with respective maxitive integral  
	$\phi_J(f)=\limsup_{n\to\infty}\tfrac{1}{n}\log\mathbb{E}_{\mathbb{P}}[\exp(n f(X_n))]$ for all bounded continuous functions $f$ on $E$. Then the equivalences \eqref{bound1} and \eqref{bound2}  amount to the well-known equivalence between the large deviation principle (LDP) and the Laplace principle (LP); cf.~\cite{dembo,puhalskii2001large}. In this article, we will show that the key concepts of large deviations theory can be understood and extended to the framework of 
	weakly maxitive concentrations and their maxitive integrals; e.g., the equivalences \eqref{bound1} and \eqref{bound2} establish the equivalence between the LDP and the LP for general concentrations. This covers situations that are not captured by the standard setting of large deviations theory. 
	 For instance, in the theory of imprecise probability~\cite{walley1991statistical}, we may be interested in large deviations bounds for upper probabilities of the form $\overline{\mathbb{P}}(A)=\sup_{\mathbb{P}\in\mathcal{P}}\mathbb{P}(A)$ for a set $\mathcal{P}$ of probability measures.
	Then we consider the concentration $J_A=\limsup_{n\to\infty}\tfrac{1}{n}\log\overline{\mathbb{P}}(X_n\in A)$, for which we will show
	in Subsection~\ref{sec:acc} that its maxitive integral is given by
	\[
	\phi_J(f)=\limsup_{n\to\infty}\tfrac{1}{n}\log\int_0^\infty \overline{\mathbb{P}}\left(\exp(nf(X_n))> x\right) {\rm d}x.
	\]
	Standard large deviations theory provides conditions such that the usual upper bound $\limsup_{n\to\infty}\tfrac{1}{n}\log \mathbb{P}(X_n\in C)\le -\inf_{x\in C}I(x)$ is valid for all closed sets $C\subset E$, and the usual lower bound $\liminf_{n\to\infty}\tfrac{1}{n}\log \mathbb{P}(X_n\in O)\ge -\inf_{x\in O}I(x)$ holds for all open sets $O\subset E$. 
	However, one may be interested in finding bounds on certain smaller classes of sets.   
	One of the features of the presented framework is that, by considering a preorder relation $\le$ on $E$, we may restrict ourselves to the class of upwards closed sets 
	for which we characterize large deviation bounds.  As a result, we obtain a monotone version of Cram\'{e}r's theorem which provides new large deviations bounds for the sample mean of i.i.d.~sequences.

\section{Bounds for maxitive integral and weak maxitivity}\label{sec:maxInt}
In this section, we introduce the basic concepts and provide the key bounds for concentrations and the respective maxitive integrals. 
These bounds ensure a weak form of maxitivity which allows to connect concentrations and their maxitive integrals with rate functions.

\subsection{Setting and notation} Let $(E,\le)$ be a topological preordered space.\footnote{Recall that a preorder is a reflexive and transitive binary relation.  We do not assume any relations between
	the topology and the preorder.}  Let $\mathcal{U}$ be a base of the topology, and define $\mathcal{U}_x:=\{U\in\mathcal{U}\colon x\in U\}$ for all $x\in E$.
 Moreover, for $A\subset E$, we define the \emph{upset} and the \emph{downset} as
\[
{\uparrow}A:=\{y\in E\colon x\le y\mbox{ for some }x\in A\}\quad\mbox{ and }\quad{\downarrow}A:=\{y\in E\colon y\le x\mbox{ for some }x\in A\}.
\]  
We say that $A\subset E$ is \emph{upwards closed} if $A={\uparrow}A$, and \emph{downwards closed} if $A={\downarrow}A$.   
Let $\mathcal{O}^{\uparrow}$ denote the collection of all subsets $A\subset E$ which are open and upwards closed, and $\mathcal{C}^{\uparrow}$ be the collection of all  subsets $A\subset E$ which are closed and upwards closed. In addition, let $\mathcal{C}_c^\uparrow$ be the set of all $C\in \mathcal{C}^\uparrow$ which are compactly generated, i.e., $C={\uparrow}K$ for some compact $K\subset E$.
 Similarly, we define the corresponding collections $\mathcal{O}^{\downarrow}$ and $\mathcal{C}^{\downarrow}$ of downwards closed sets, which are open and closed, respectively. 

Throughout this section, we work under the following assumption.
\begin{assumption}\label{ass1}
For every $x,y\in E$ with $x\le y$, we assume that 
\begin{itemize}
	\item[(A)] for every $U_x\in\mathcal{U}_x$ there exists $U_y\in\mathcal{U}_y$ such that $U_y\subset {\uparrow}U_x$.
	\item[(B)] for every $U_y\in\mathcal{U}_y$ there exists $U_x\in\mathcal{U}_x$ such that $U_x\subset {\downarrow}U_y$.
\end{itemize}
\end{assumption}
In case that the preorder is trivial, i.e., $x\le y$ if and only if $x=y$, then  $\mathcal{O}^{\uparrow}$ and $\mathcal{C}^{\uparrow}$ coincide with the collections of all open and closed subsets of $E$, respectively, and the previous assumption is trivially satisfied. 

\begin{remark}\label{rem:preorderedGroup}
Let $E$ be a Hausdorff topological abelian group. Suppose that $E_+$ is a subset of $E$ such that $E_+ + E_+\subset E_+$ and $0\in E_+$. 
	Then the binary relation defined by 
	$$x\le y\quad\mbox{if and only if}\quad y-x\in E_+$$  
	is a translation invariant preorder on $E$. 
	Direct verification shows that $(E,\le)$ satisfies Assumption~\ref{ass1}.
\end{remark}

 On a functional level, the sets  $\mathcal{O}^{\uparrow}$ and $\mathcal{C}^{\uparrow}$ correspond to the following spaces.
We denote by ${L}^{\uparrow}$ the set of all increasing\footnote{A function $f\colon E\to [-\infty,\infty]$ is called \emph{increasing}  if $f(x)\le f(y)$  whenever $x\le y$.} lower semicontinuous functions $f\colon E\to[-\infty,\infty)$, by ${U}^{\uparrow}$ the set of all increasing upper semicontinuous functions $f\colon E\to[-\infty,\infty)$, and 
by $U_c^\uparrow$ the set of all $f\in U^\uparrow$ such that $\{f\ge c\}$ is compactly generated for all $c\in\R$. 
We first collect some basic topological properties.
\begin{lemma}\label{lem:basictop}
	 The following assertions hold.
	\begin{enumerate}
		\item $A\in \mathcal{O}^\uparrow$ if and only if $A^c\in \mathcal{C}^\downarrow$. 
		\item $A\in \mathcal{C}^\uparrow$ if and only if $A^c\in \mathcal{O}^\downarrow$. 
		\item If $f\in {L}^{\uparrow}$, then $\{f>c\}
		\in \mathcal{O}^{\uparrow}$  for all $c\in\R$.  
		\item 	If $f\in {U}^{\uparrow}$, then $\{f\ge c\}
		\in \mathcal{C}^{\uparrow}$ for all $c\in\R$.
		\item 	If $U\subset E$ is open, then ${\uparrow}U\in\mathcal{O}^{\uparrow}$. 
		\item 	If $A$ is upwards closed, then ${\rm cl}(A)$ is upwards closed.\footnote{As usual, ${\rm cl}(A)$ denotes the topological closure of a subset $A$.}
	\end{enumerate}
\end{lemma} 

\begin{proof}
	(i) Suppose that $A\in\mathcal{O}^\uparrow$. 
	We show that $A^c$ is downwards closed.  
	By contradiction, assume that $A^c\neq {\downarrow}(A^c)$, that is, there exists $x\in {\downarrow}(A^c)$ such that $x\notin A^c$ (or equivalently $x\in A$). 
	Since $x\in {\downarrow}(A^c)$, there exists $y\in A^c$ such that $x\le y$. 
	Since $x\in A$ and $x\le y$, it follows that $y\in {\uparrow}A$. 
	By assumption ${\uparrow}A=A$, so that $y\in A$. 
	But this is a contradiction to $y\in A^c$. 
	
	(ii) follows along the same line of argumentation as (i).
	
	(iii) Fix $f\in L^{\uparrow}$. 
	We prove that $\{f>c\}$ is upwards closed. 
	By contradiction, assume that $x\le y$ for $x\in \{f>c\}$ and $y\notin \{f>c\}$. 
	Since $x\in \{f>c\}$, it follows that $f(x)> c$. 
	Since $y\notin  \{f>c\}$, we have that $f(y)\le c$. 
	Hence $f(y)<f(x)$, in contradiction to $f(x)\le f(y)$.  
	
	(iv) follows by similar arguments as in (iii).
	
	(v) Let $y\in {\uparrow}U$, so that $x\le y$ for some  $x\in U$. Since $U$ is open, there exists $U_x\in\mathcal{U}_x$ such that $U_x\subset U$.
	By Assumption \ref{ass1}, there exists $U_y\in\mathcal{U}_y$ such that $U_y\subset{\uparrow}U_x\subset{\uparrow} U$. This shows that ${\uparrow}U$ is open.
	
	(vi) Suppose that $x\le y$ with $x\in {\rm cl}(A)$.
	We have to show that $y\in {\rm cl}(A)$. Fix $U_y\in\mathcal{U}_y$. By Assumption \ref{ass1} there exists $U_x\in\mathcal{U}_x$ such that 
	$U_x\subset {\downarrow}U_y$. Since $x$ is in the closure of $A$, there exists $\tilde{x}\in U_x\cap A$.
	Therefore, it follows that $\tilde{x}\in {\downarrow}U_y$, which shows that there exists 
	$\tilde{y}\in U_y$ with $\tilde{x}\le \tilde{y}$. Since $\tilde{y}\in {\uparrow} A=A$, we conclude that  $U_y\cap A\neq\emptyset$.
\end{proof} 

\subsection{Bounds for maxitive integrals} In accordance with Section~\ref{sec:possdistribution}, we next introduce the key concepts of this article. Let $J$ be a \emph{concentration} on $\mathcal{OC}^{\uparrow}:=\mathcal{O}^{\uparrow}\cup \mathcal{C}^{\uparrow}$, i.e., a set function $J\colon \mathcal{OC}^{\uparrow}\to [-\infty,0]$ which satisfies $J_\emptyset=-\infty$, $J_E=0$, and $J_A\le J_B$ whenever $A\subset B$. 
The respective \emph{maxitive integral}  $\phi_J$ on ${LU}^{\uparrow}:={L}^{\uparrow}\cup{U}^{\uparrow}$ is defined by
\begin{equation}\label{def:maxint}
\phi_J(f):=\begin{cases}
	\underset{c\in\R}\sup\{c + J_{\{f> c\}}\}, & \mbox{ if }f\in {L}^\uparrow,\\
	\underset{c\in\R}\sup\{c + J_{\{f\ge c\}}\}, & \mbox{ otherwise}.
\end{cases}
\end{equation}
As discussed in the previous section, $e^J$ is a capacity and $\phi_J$ is a transformed version of the Shilkret integral; cf.~Cattaneo~\cite{cattaneo2016maxitive}.
Notice that $\phi_J$ is well-defined due to Lemma~\ref{lem:basictop}. By considering the extended concentration $\bar J_A:=\inf\{J_B\colon B\in  \mathcal{OC}^{\uparrow},\,A\subset B \}$ for all $A\subset E$, it follows by direct verification that for all $f\in {LU}^{\uparrow}$,
\[
	\phi_J(f)=\underset{c\in\R}\sup\{c + \bar J_{\{f\ge c\}}\}=\underset{c\in\R}\sup\{c + \bar J_{\{f> c\}}\}.
\]
This shows that definition \eqref{def:maxint} is consistent with that in \eqref{def:convint}.\footnote{In particular, for $f\in   {L}^{\uparrow}\cap  {U}^{\uparrow}$, it follows that the two definitions in \eqref{def:maxint} coincide.} The functional  \eqref{def:maxint} shares the properties of a non-linear expectation, i.e., it is constant preserving $\phi_J(c)=c$ for all $c\in \R$, and monotone $\phi_J(f)\le\phi_J(g)$ whenever $f\le g$. Also, the maxitive integral has the translation property $\phi_J(f+c)=\phi_J(f)+c$ for all $c\in\R$. 

Next, we show that the concentration can be recovered from the maxitive integral by the evaluation at indicator functions. 
We always make the convention that $-\infty \cdot 0=0$, so that the indicator function $-\infty 1_{A^c}$ assumes the value $-\infty$ on $A^c$ and zero on $A$. 

\begin{proposition}\label{prop:concentration} Let $J$ be a concentration. Then, for every $A\in \mathcal{OC}^{\uparrow}$, 
\[
J_A=\phi_J(-\infty 1_{A^c})=\underset{r<0}\inf \phi_J(r 1_{A^c}).
\]
\end{proposition}
\begin{proof}

Suppose, for instance, that $A\in \mathcal{O}^{\uparrow}$.
 Then, $A^c\in \mathcal{C}^{\downarrow}$ due to Lemma \ref{lem:basictop}, and therefore $-\infty 1_{A^c}\in {L}^{\uparrow}$. 
 We obtain  
\begin{align*}
\phi_J(-\infty 1_{A^c})=\underset{c\in\R}\sup\{c + J_{\{-\infty 1_{A^c}>c\}}\}=\underset{0\le c}\sup\{c + J_{\emptyset}\}\vee \underset{c<0}\sup\{c + J_{A}\}=J_A.
\end{align*}
Now, let $r<0$, so that $r 1_{A^c}\in {L}^{\uparrow}$. Then,
\begin{align*}
\phi(r 1_{A^c})&=\underset{c\in\R}\sup\{c + J_{\{r 1_{A^c}>c\}}\}\\
&=\underset{c< r}\sup\{c + J_{\{r 1_{A^c}>c\}}\}\vee\underset{r\le c< 0}\sup\{c + J_{\{r 1_{A^c}>c\}}\}\vee\underset{0\le c}\sup\{c + J_{\{r 1_{A^c}>c\}}\}\\
&=r \vee J_{A} \vee (-\infty) =r \vee J_{A}.
\end{align*}
Hence, by letting $r\to-\infty$, we conclude $\underset{r<0}\inf \phi_J(r 1_{A^c})=J_A$.  
\end{proof}
In accordance with Section \ref{sec:possdistribution}, we say that a concentration $J$ admits a rate function if there exists a function $I\colon E\to[0,\infty]$ such that $J_A=-\inf_{x\in A} I(x)$ for all $A\in  \mathcal{OC}^{\uparrow}$.
In that case, the concentration is completely maxitive in the sense that $J_A\le \vee_{i\in {\mathcal I}}J_{A_i}$ for every family $(A_i)_{i\in {\mathcal I}}\subset \mathcal{OC}^{\uparrow}$ 
and $A\in \mathcal{OC}^{\uparrow}$  with $A\subset \cup_{i\in {\mathcal I}} A_i$, and the maxitive integral  $\phi_J(f)$ admits the representation $\sup_{x\in E}\{f(x)-I(x)\}$ for all $f\in {LU}^\uparrow$.
In the following, we relax the relation between concentrations and rate functions. 
As a first main result, we obtain that the maxitive integral satisfies the following upper and lower bounds.
\begin{theorem}\label{thm:mVaradhan1}
	Let $J$ be a concentration and $I\colon E\to[0,\infty]$ be a function.  
	Then, the following equivalences hold. First,
	\begin{equation}\label{eq:LDPL}
		-\inf_{x\in O}I(x) \le J_{O}\quad \mbox{for all }O\in\mathcal{O}^\uparrow
	\end{equation}
	if and only if 
	\begin{equation}\label{eq:L1}
		\phi_J(f)\ge \underset{x\in E}\sup\{f(x)-I(x)\}\quad\mbox{ for all }f\in L^\uparrow.
	\end{equation}	
	Second,
	\begin{equation}\label{eq:LDPU}
		J_C\le -\inf_{x\in C}I(x)\quad \mbox{for all }C\in\mathcal{C}^\uparrow	
	\end{equation}
	if and only if 	
	\begin{equation}\label{eq:U1}
		\phi_J(f)\le \underset{x\in E}\sup\{f(x)-I(x)\}\quad\mbox{ for all }f\in U^\uparrow. 
	\end{equation} 
	Third,
	\begin{equation}\label{eq:LDPUc}
		J_C\le -\inf_{x\in C}I(x)\quad \mbox{for all }C\in\mathcal{C}_c^\uparrow
	\end{equation}
	if and only if 	
	\begin{equation}\label{eq:U1c}
		\phi_J(f)\le \underset{x\in E}\sup\{f(x)-I(x)\}\quad\mbox{ for all }f\in U_c^\uparrow. 
	\end{equation} 
\end{theorem}
\begin{proof}
	Suppose that inequality \eqref{eq:LDPL} holds. Let $f\in L^\uparrow$ and  $\varepsilon>0$. Using the definition of the maxitive integral $\phi_J$ and the inequality  $J_{\{f>r\}}\ge -\inf_{x\in \{f>r\}} I(x)$ for all $r\in\R$, 
	\begin{align*}
		\phi_J(f)&\ge \underset{r\in\R}\sup\big\{r -\inf_{x\in \{f>r\}} I(x)\big\}\\
		&=\underset{r\in\R}\sup\;\underset{x\in \{f> r\}}\sup\{r -I(x)\}\\
		&\ge \underset{x\in E}\sup\;\underset{y\in \{f>f(x)-\varepsilon\}}\sup\{f(x)-\varepsilon -I(y)\}\\
		&\ge \underset{x\in E}\sup\{f(x)-\varepsilon -I(x)\}.
	\end{align*}
	Since $\varepsilon>0$ was arbitrary, we obtain $\phi_J(f)\ge\sup_{x\in E}\{f(x) -I(x)\}$.
	
	Conversely, suppose that inequality \eqref{eq:L1} holds. Let $O\in \mathcal{O}^\uparrow$, so that $-\infty 1_{O^c}\in L^\uparrow$. 
	Then, by Proposition \ref{prop:concentration}, 
	\[
	J_{O}=\phi_J(-\infty 1_{O^c})\ge \underset{x\in E}\sup\{-\infty 1_{O^c}(x)-I(x)\}=-\underset{x\in O}\inf I(x).
	\] 
	
	Now, suppose that inequality \eqref{eq:LDPU} holds. Let $f\in U^\uparrow$. Using the definition of the maxitive integral $\phi_J$ and the inequality 
	$J_{\{f\ge r\}}\le -\inf_{x\in \{f\ge r\}} I(x)$ for all $r\in\mathbb{R}$, 
	\begin{align*}
		\phi_J(f)&\le \underset{r\in\R}\sup\big\{r -\inf_{x\in \{f\ge r\}} I(x)\big\}\\
		&=\underset{r\in\R}\sup\;\underset{x\in \{f\ge r\}}\sup\{r -I(x)\}\\
		&\le\underset{r\in\R}\sup\;\underset{x\in \{f\ge r\}}\sup\{f(x) -I(x)\}\\
		&\le \underset{x\in E}\sup\{f(x) -I(x)\}.
	\end{align*}
	
	Conversely, suppose that inequality \eqref{eq:U1} holds. Let $C\in \mathcal{C}^\uparrow$, so that  $-\infty 1_{C^c}\in U^\uparrow$. 
	It follows from Proposition \ref{prop:concentration} that
	\[
	J_{C}=\phi_J(-\infty 1_{C^c})\le \underset{x\in E}\sup\{-\infty 1_{C^c}(x)-I(x)\}=-\underset{x\in C}\inf I(x).
	\]
	Finally, the equivalence between the inequalities \eqref{eq:LDPUc} and \eqref{eq:U1c} follows along the same line of argumentation by replacing $\mathcal{C}^\uparrow$ by $\mathcal{C}_c^\uparrow$ and $U^\uparrow$ by $U_c^\uparrow$, respectively.
\end{proof}
Notice that the upper/lower bound in the previous result are specified by upper/lower semicontinuous functions. In the same spirit, in the context of the Hausdorff moment problem \cite{miranda}, the natural extension of the moment sequence is also determined by semicontinuous functions.

Given a concentration $J$, our goal now is to find a rate function $I\colon E\to[0,\infty]$ which satisfies the lower bound \eqref{eq:LDPL} and upper bounds \eqref{eq:LDPU} or \eqref{eq:LDPUc}. We first focus on the lower bound. From the previous theorem, we immediately see that 
\[
I_{\min}(x):= \sup_{f\in  L^\uparrow} (f(x)-\phi_J(f))\quad\mbox{for all } x\in E
\]
satisfies $I_{\min}(x)\ge f(x)-\phi_J(f)$ for all $x\in E$ and $f\in L^\uparrow$, and therefore the inequalities \eqref{eq:L1} and \eqref{eq:LDPL}. 
We refer to $I_{\min}$ as the \emph{minimal rate function}. Notice that  $I_{\min}$ is minimal in the class of functions $I\colon E\to[0,\infty]$ which satisfy the inequalities \eqref{eq:LDPL} and \eqref{eq:L1}.\footnote{In fact, if $I\colon E\to[0,\infty]$ is another function for which inequality \eqref{eq:L1} holds, then $I(x)\ge f(x)-\phi_J(f)$ for all $x\in E$ and $f\in L^\uparrow$, and therefore  $I(x)\ge \sup_{f\in  L^\uparrow} (f(x)-\phi_J(f))=I_{\min}(x)$ for all $x\in E$.} By definition, the function $I_{\rm min}$ is lower semicontinuous and increasing. Moreover, it is determined through the concentration function as follows.
\begin{lemma}\label{lem:ratefunction1} 
Let $I_{\rm min}$ be the minimal rate function associated to a concentration $J$. Then, for every $x\in E$, 
\[
-I_{\rm min}(x)=\underset{U\in \mathcal{U}_x}\inf J_{{\uparrow}U}.
\]
\end{lemma}
\begin{proof}
Let $x\in E$ and $U\in\mathcal{U}_x$.  
By Lemma \ref{lem:basictop}, we have $({\uparrow}U)^c\in \mathcal{C}^{\downarrow}$, and therefore  $-\infty 1_{({\uparrow}U)^c}\in {L}^{\uparrow}$.
Using the definition of the minimal rate function $I_{\rm min}$ and Proposition~\ref{prop:concentration}, we obtain $I_{\rm min}(x)\ge -\phi_J(-\infty 1_{({\uparrow}U)^c})= -J_{{\uparrow}U}$.
Hence, by taking the supremum over all $U\in\mathcal{U}_x$, 
\[
I_{\rm min}(x)\ge -\underset{U\in\mathcal{U}_x}\inf J_{{\uparrow}U}.
\]
As for the other inequality, let $f\in L^{\uparrow}$ and $\varepsilon>0$.
Since $f$ is lower semicontinuous, there exists $U\in\mathcal{U}_x$ such that $f(z)\ge f(x)-\varepsilon$  for all $z\in U$.
In particular, $f(z)\ge f(x)-\varepsilon$ for all $z\in {\uparrow}U$ because  $f$ is increasing. Then, by monotonicity and the translation property of $\phi_J$ and Proposition~\ref{prop:concentration}, 
\begin{align*}
	\phi_J(f)&\ge \phi_J( (f(x)-\varepsilon)1_{{\uparrow}U} - \infty 1_{({\uparrow}U)^c} )\\
	&=f(x)-\varepsilon + \phi_J(-\infty  1_{({\uparrow}U)^c})\\
	&= f(x)-\varepsilon + J_{{\uparrow}U}.
\end{align*}
Since $\varepsilon>0$ was arbitrary, we get $\phi_J(f)\ge f(x) + J_{{\uparrow}U}\ge f(x) + \inf_{U\in\mathcal{U}_x}J_{{\uparrow}U}$.
This shows that for all $f\in L^{\uparrow}$,
\[
\underset{U\in\mathcal{U}_x}\inf J_{{\uparrow}U}\le  -f(x) + \phi_J(f).
\] 
By taking the infimum over all $f\in L^{\uparrow}$, we conclude $\underset{U\in \mathcal{U}_x}\inf J_{{\uparrow}U}\le  -I_{\rm min}(x)$.
\end{proof}
For the minimal rate function to satisfy the upper bound \eqref{eq:LDPU} or \eqref{eq:LDPUc}, additional assumptions on the concentration are required. 
We will see that the concentration must necessarily fulfil a certain form of maxitivity if the lower and upper bounds simultaneously hold. 
This is the context of the next subsection.

\subsection{Weak maxitivity} As discussed in Section~\ref{sec:possdistribution}, a capacity which admits a possibility distribution in the sense of \eqref{eq:possibilityIntegral} (respectively a concentration which admits rate function) is completely maxitive.  In case that inequality~\eqref{eq:possibilityMeasure} is only assumed for sequences or a finite sets, then the capacity is called countably maxitive or finitely maxitive, respectively. For a detailed discussion on different maxitivity concepts, we refer to~\cite{cattaneo2016maxitive}.
In case that a concentration satisfies only upper and lower bounds, we obtain a weaker form of maxitivity.
\begin{definition}
	A concentration function $J$ is called \emph{weakly maxitive} if 
	\[
	J_A\le \vee_{i=1}^n J_{B_i}\quad
	\mbox{ for all } A\in \mathcal{C}^{\uparrow}\mbox{ and }B_1,\dots,B_n\in  \mathcal{O}^{\uparrow},\,n\in\N, \mbox{ such that }A\subset \cup_{i=1}^n B_{i}.
	\]
Likewise,  a function $\phi\colon {LU}^\uparrow\to [-\infty,\infty]$ is called \emph{weakly maxitive}  if 
	\[
	\phi(f)\le \vee_{i=1}^n \phi(g_i)
	\mbox{ for all }f\in {U}^\uparrow\mbox{ and }g_1,\dots,g_n\in {L}^\uparrow,\,n\in\N,\mbox{ such that }f\le \vee_{i=1}^n g_i.
	\]
\end{definition}

As a direct consequence of Theorem~\ref{thm:mVaradhan1}, we obtain the following result.
\begin{proposition}
	Suppose that a concentration $J$ satisfies the inequalities \eqref{eq:LDPL} and \eqref{eq:LDPU} for a
	function $I\colon E\to [0,\infty]$. Then, $J$ and $\phi_J$ are weakly maxitive.
\end{proposition}
\begin{proof}
	Let $A\in\mathcal{C}^\uparrow$ and $B_1,\dots,B_n\in  \mathcal{O}^{\uparrow}$, $n\in\N$, with $A\subset \cup_{i=1}^n B_i$. 
	Then, by monotonicity of $J$ and the inequalities   \eqref{eq:LDPU} and \eqref{eq:LDPL},
	\[
	J_A\le-\inf_{x\in A} I(x)\le-\inf_{x\in \cup_{i=1}^n B_i} I(x)=-\wedge_{i=1}^n \inf_{x\in B_i} I(x)\le \vee_{i=1}^n J_{B_i}.
	\]	
	Furthermore, let $f\in U^\uparrow$ and $g_1,\dots,g_n\in L^\uparrow$, $n\in\N$, with $f\le \vee_{i=1}^n g_i$. Then, it follows from Theorem~\ref{thm:mVaradhan1} that
	\begin{align*}
		\phi_J(f)&\le \underset{x\in E}\sup\{f(x)-I(x)\}\\
		&\le \underset{x\in E}\sup\{\vee_{i=1}^n g_i(x)-I(x)\}\\
		&\le \vee_{i=1}^n \underset{x\in E}\sup\{g_i(x)-I(x)\}\\
		&\le \vee_{i=1}^n \phi_J(g_i).\qedhere
	\end{align*} 
\end{proof}
We now turn to the converse question and provide bounds for weakly maxitive concentrations.
Thereby, we focus on the upper bound.
 \begin{proposition}\label{prop:LDPupper}
	Suppose that $J$ is weakly maxitive. Then, $J$ satisfies the upper bound~\eqref{eq:LDPUc} for the minimal rate function $I_{\rm min}$.
\end{proposition}
\begin{proof}
	Let $C\in\mathcal{C}_c^\uparrow$. By definition, $C\in\mathcal{C}^\uparrow$ and $C={\uparrow}K$ for some compact $K\subset E$. Fix $\varepsilon>0$. Then, by Lemma~\ref{lem:ratefunction1} and compactness, there exist $x_i\in K$ and $U_i\in\mathcal{U}_{x_i}$ for $i=1,\ldots,n$  such that $K\subset \bigcup_{i=1}^n U_i$ and 
	\[
	-J_{{\uparrow}U_i}\ge (I_{\rm min}(x_i)-\varepsilon)\wedge \varepsilon^{-1}\quad \mbox{ for all }i=1,\ldots,n.
	\]
	Consequently, since ${\uparrow}U_i\in \mathcal{O}^\uparrow$ for all $i=1,\dots, n$ due to Lemma~\ref{lem:basictop},  $ C\subset \bigcup_{i=1}^n {\uparrow}U_i$,
	and $J$ is weakly maxitive, 
	\[
	-J_{C} \ge \wedge_{i=1}^n -J_{{\uparrow}U_i}\ge \wedge_{i=1}^n (I_{\rm min}(x_i)-\varepsilon)\wedge \varepsilon^{-1}\ge  (\underset{x\in K}\inf I_{\rm min}(x)-\varepsilon)\wedge \varepsilon^{-1}. 
	\]
	Since $\varepsilon>0$ was arbitrary, we obtain that 
	\[
	-J_{C}\ge \inf_{x\in K} I_{\rm min}(x)=\inf_{x\in C}\I(x),
	\]
	where the last equality holds because $\I$ is increasing.  
\end{proof}
For the minimal rate function to satisfy the stronger upper bound \eqref{eq:LDPU}, we need an additional tightness assumption.
\begin{definition}
	A concentration  $J$ is called \emph{tight}, if for all $C\in \mathcal{C}^\uparrow$ and 
	$\varepsilon>0$, there exists $K\subset E$ compact with ${\uparrow}(C\cap K)\in  \mathcal{C}^\uparrow_c$ such that
	\begin{equation}\label{eq:tightness0}
		J_C\le \left(J_{{\uparrow}(C\cap K)}+\varepsilon\right)\vee(-\varepsilon^{-1}).
	\end{equation}
Notice that $C\cap K$ is compact. In typical situations,  ${\uparrow}(C\cap K)$ is closed, see Lemma~\ref{lem:compactUp} in Appendix A.
\end{definition} 

\begin{corollary}\label{cor:tightness}
	Let $J$ be an weakly maxitive concentration which is tight. Then, $J$ satisfies the upper bound \eqref{eq:LDPU} with minimal rate function $I_{\rm min}$.
\end{corollary}
\begin{proof}
	Let $C\in\mathcal{C}^\uparrow$ and $\varepsilon>0$. Since $J$ is tight, there exists $K\subset E$ compact with ${\uparrow}(C\cap K)\in  \mathcal{C}^\uparrow_c$ such that inequality \eqref{eq:tightness0} holds. Hence, we can apply Proposition \ref{prop:LDPupper} to obtain
	\begin{align*}
		J_C
		&\le \left(J_{{\uparrow}(C\cap K)}+\varepsilon\right)\vee (-\varepsilon^{-1}) \le \left(-\underset{x\in {\uparrow}(C\cap K)}\inf I_{\rm min}(x)+\varepsilon\right)\vee(-\varepsilon^{-1})\\ 
		&\le \left(-\underset{x\in C}\inf I_{\rm min}(x)+\varepsilon\right)\vee(-\varepsilon^{-1}).
	\end{align*}
	Since $\varepsilon>0$ was arbitrary, we conclude that $J_C\le -\inf_{x\in C} I_{\rm min}(x)$. 
\end{proof}
The results achieved so far can be summarized as follows. For a weakly maxitive concentration $J$ there exists a function $I\colon E\to[0,\infty]$ such that 
\begin{equation}\label{mot:bounds}
-\inf_{x\in O} I(x)\le J_O\quad\mbox{and}\quad J_C\le-\inf_{x\in C} I(x)
\end{equation}
for all $O\in \mathcal{O}^\uparrow$ and $C\in  \mathcal{C}^\uparrow_c$, respectively $C\in  \mathcal{C}^\uparrow$ if the concentration is in addition tight. In this case, 
the respective maxitive integral is given by $\phi_J(f)=\sup_{x\in E}(f(x)-I(x))$ for all $f\in L^\uparrow\cap U^\uparrow_c$ and $f\in L^\uparrow\cap U^\uparrow$, respectively. Moreover, the function $I$ can be replaced by the minimal rate function $I_{\rm min}$. In other words, the class of weakly maxitive concentrations (capacities) can be connected with a rate function (possibility distribution), which fully determines the associated maxitive integral for sufficiently regular functions. To what extent the rate function is unique will be discussed in Section~\ref{sec:monLDP},
and how it can be determined by means of convex duality arguments in Section~\ref{sec:convRateFunc}. 

\section{Representation of weakly maxitive non-linear expectations}\label{sec:integralRep}
As remarked in the previous section, the maxitive integral  \eqref{def:maxint} has the same properties as a non-linear expectation which satisfies the translation property.
In the following, we focus on the converse direction and investigate when a non-linear expectation with the translation property can be represented as a  maxitive integral. 
To do so, the following continuity condition is necessary.
\begin{lemma}\label{lem:upperextension}
	For every $f\in {LU}^{\uparrow}$, 
	$$\underset{N\to\infty}\lim \phi_J(f\wedge N)=\phi_J(f)\quad\mbox{and}\quad\underset{N\to\infty}\lim \phi_J(f\vee(-N))=\phi_J(f).$$
\end{lemma}
\begin{proof}
	Let $f\in {L}^{\uparrow}$, the other case follows with similar arguments. By definition of the maxitive integral, for each $N\in\N$, 
	\begin{align*}
		\phi_J(f\wedge N)&=\underset{c\in \R}\sup\{c + J_{\{(f\wedge N)>c\}}\}
		=\underset{c< N}\sup\{c + J_{\{f>c\}}\}\vee \underset{N\le c}\sup\{c + J_\emptyset\}
		=\underset{c< N}\sup\{c + J_{\{f>c\}}\},
	\end{align*}
	 and therefore, 
	\[
	\underset{N\to\infty}\lim \phi_J(f\wedge N)=\underset{N\in\N}\sup\phi_J(f\wedge N)
	=\sup_{N\in\N} \sup_{c< N} \{c + J_{\{f>c\}}\}=\underset{c\in\R}\sup\{c + J_{\{f>c\}}\}=\phi_J(f).
	\]
	As for the second statement, for each $N\in\N$,  
	\begin{align*}
		\phi(f\vee (-N))&=\underset{c\in \R}\sup\{c + J_{\{f\vee (-N)>c\}}\}\\
		&=\underset{c< -N}\sup\{c + J_{\{f\vee (-N)>c\}}\}\vee \underset{-N\le c}\sup\{c + J_{\{f\vee (-N)>c\}}\}\\
		&=\underset{c< -N}\sup\{c + J_{E}\}\vee \underset{-N\le c}\sup\{c + J_{\{f>c\}}\}\\
		&\le(-N)\vee \underset{c\in\R}\sup\{c + J_{\{f>c\}}\}\\
		&= (-N)\vee \phi_J(f).
	\end{align*}
    Hence, by monotonicity of $\phi_J$, 
	\[
	\phi_J(f)\le\underset{N\to\infty}\lim \phi_J(f\vee (-N))\le \underset{N\to\infty}\lim (-N)\vee \phi_J(f)=\phi_J(f). \qedhere
	\]
\end{proof}
As shown in \cite[Corollary 6]{cattaneo2016maxitive}, every finitely maxitive non-linear expectation with the translation property admits a representation in terms of a maxitive integral; see also~\cite[Corollary 7]{cattaneo2016maxitive} and \cite[Proposition 2.2]{kupper}.   
We next provide a related representation result for weakly maxitive non-linear expectations.
Let  $\overline{L}^\uparrow$, $\overline{U}^\uparrow$, and $\overline{LU}^\uparrow$  denote the sets of all functions in ${L}^\uparrow$, ${U}^\uparrow$, and ${LU}^\uparrow$, respectively, which are bounded from above. Moreover, let $C_b^{\uparrow}$ be the space of all increasing bounded continuous functions $f\colon E\to\R$.
\begin{theorem}\label{thm:maxitiveRep}
	Suppose that $\psi\colon \overline{LU}^{\uparrow}\to [-\infty,\infty)$ satisfies
	\begin{enumerate}
		\item $\psi(0)=0$,
		\item $\psi(f)\le\psi(g)$ whenever~$f\le g$,
		\item $\psi(f+c)=\psi(f)+c$ for all $c\in\R$,
	\end{enumerate}
with concentration $J^\psi_A:=\psi(-\infty 1_{A^c})$ for all $A\in \mathcal{OC}^{\uparrow}$. 
	If $\psi$ is weakly maxitive,\footnote{I.e.,~$\phi(f)\le \vee_{i=1}^n \phi(g_i)$ for all $f\in \overline{U}^\uparrow$ and $g_1,\dots,g_n\in \overline{L}^\uparrow$, $n\in\N$, such that $f\le \vee_{i=1}^n g_i$.} 
	then
	\[	\psi(f)=	\phi_{J^\psi}(f)\quad\mbox{for all }f\in \overline{L}^{\uparrow}\cap \overline{U}^{\uparrow}.\]
\end{theorem}
\begin{proof} 
	First, suppose that $f\in C^{\uparrow}_b$ and let $a,b\in\R$ such that $a<f(x)<b$ for all $x\in E$. 
	Fix $N\in\mathbb{N}$, and define for $0\le j\le N-1$,
		\[
	a_{N,j}:=a + j\frac{b-a}{N},\quad O_{N,j}=\{f>a_{N,j}\},\quad\mbox{and}\quad C_{N,j}=\{f\ge a_{N,j}\}.
	\]
	We consider the simple functions
	\[
	l_N:=\vee^{N-1}_{j=0}\big(- \infty 1_{O_{N,j}^c}+a_{N,j} 1_{O_{N,j}} \big)\quad\mbox{ and }\quad 
	u_N:=\vee^{N-1}_{j=0}\big(- \infty 1_{C_{N,j}^c}+a_{N,j} 1_{C_{N,j}} \big).   
	\] 
	By construction, it holds $l_N\in \overline{L}^{\uparrow}$, $u_N\in \overline{U}^{\uparrow}$, and $f-\frac{b-a}{N}\le l_N\le u_N\le f$.
	Using the definition of the maxitive integral and the translation property of $\psi$,
	\begin{align*}
		\phi_{J^\psi}(l_N)
		&=\vee_{j=0}^{N-1}\{a_{N,j} + J_{O_{N,j}}^\psi\}\\
		&=\vee_{j=0}^{N-1}\psi\big( - \infty 1_{O_{N,j}^c}+a_{N,j}1_{O_{N,j}}\big)\\
		&\le \psi(l_N).
	\end{align*}
	Moreover, it follows from 
	$$u_N\le l_N + \frac{b-a}{N} =\vee^{N-1}_{j=0}\big(- \infty 1_{O_{N,j}^c}+\left(a_{N,j}+\tfrac{b-a}{N}\right)1_{O_{N,j}} \big),$$
	and the weak maxitivity of $\psi$ that 
	\begin{align*}
		\psi(u_N)&\le \vee_{j=0}^{N-1}\psi\Big(- \infty 1_{O_{N,j}^c}+\left(a_{N,j}+\tfrac{b-a}{N}\right)1_{O_{N,j}} \Big)\\
		&=\tfrac{b-a}{N} + \vee_{j=0}^{N-1}\psi\big(- \infty 1_{O_{N,j}^c}+a_{N,j}1_{O_{N,j}} \big)\\
		&=\tfrac{b-a}{N} + \phi_{J^\psi}(l_N).
	\end{align*}
In combination with $\phi_{J^\psi}(l_N)\le  \psi(l_N)\le  \psi(u_N)$, we obtain
	\[
	|\psi(u_N) - \phi_{J^\psi}(l_N)|\le \tfrac{b-a}{N}.
	\]	
Hence, as a consequence of the monotonicity and translation property of $\psi$ and $\phi_{J^\psi}$,
	\begin{align*}
		|\psi(f)-\phi_{J^\psi}(f)|&\le |\psi(f)-\psi(u_N)|+|\psi(u_N)-\phi_{J^\psi}(l_N)|+ |\phi_{J^\psi}(l_N) -\phi_{J^\psi}(f)|\\
		&\le \tfrac{b-a}{N} + \tfrac{b-a}{N} + \tfrac{b-a}{N}. 
	\end{align*}
	Letting $N\to\infty$ results in $\psi(f)=\phi_{J^\psi}(f)$ as desired. 
	
	Second, suppose that $f\in \overline{L}^{\uparrow}\cap \overline{U}^{\uparrow}$. 
	Let $N\in\N$, so that   $f\vee (-N)\in C^{\uparrow}_b$, and therefore $\psi(f\vee (-N))=\phi_{J^\psi}(f\vee (-N))$ due to the previous step.
	Moreover, since $\psi$ is weakly maxitive, 
	\[
		\psi(f)\le \psi(f\vee (-N))\le  \psi(f)\vee (-N).
	\]
	Hence, it follows from Lemma~\ref{lem:upperextension} that 
	\[
	\phi_{J^\psi}(f)=\lim_{N\to\infty}\phi_{J^\psi}(f\vee(-N))=\lim_{N\to\infty}\psi(f\vee(-N))=\psi(f).\qedhere
	\]
\end{proof}

\begin{remark}\label{rem:events}
A  weakly maxitive non-linear expectation with the translation property is fully determined on increasing continuous bounded functions by its restriction to the indicators $-\infty1_{A^c}$ for $A\in \mathcal{OC}^{\uparrow}$. Indeed, let $\psi_1,\psi_2\colon\overline{LU}^{\uparrow}\to [-\infty,\infty)$ be weakly maxitive non-linear expectations with the translation property. If $\psi_1(-\infty 1_{A^c})=\psi_2(-\infty 1_{A^c})$ for all $A\in \mathcal{OC}^{\uparrow}$, then Theorem~\ref{thm:maxitiveRep} implies
$\psi_1(f)=\psi_2(f)$ for all $f\in \overline{L}^{\uparrow}\cap \overline{U}^{\uparrow}$. 
\end{remark}


\section{A Laplace principle for maxitive integrals}\label{sec:monLDP} 
Theorem~\ref{thm:maxitiveRep} allows to represent weakly maxitive non-linear expectations with the translation property in terms of a maxitive integral $\phi_J$. 
In order to find a computable representation of the minimal rate function $I_{\rm min}$, we aim to find  conditions which guarantee that $I_{\rm min}$ is attained on certain spaces of bounded functions.  More specifically, under an additional separation property, we focus on representations of the form 
	\[
	I_{\rm min}(x)=\sup_{f\in C_b^{\uparrow}}\{f(x)-\phi_J(f)\}.
	\]
	We also investigate the relation between the bounds \eqref{mot:bounds} and a Laplace principle for general concentrations. Throughout this section, let $(E,\le)$ be a topological preordered space.
	We fix a base $\mathcal{U}$ for the topology of $E$, and set $\mathcal{U}_x:=\{U\in\mathcal{U}\colon x\in U\}$ for all $x\in E$. 
	Additionally to Assumption~\ref{ass1}, we require the following separation property.
\begin{assumption}\label{ass:sep}
	For every $A\in \mathcal{C}^{\uparrow}$ and $x\notin A$, there exists an increasing continuous function $f\colon E\to [0,1]$ which satisfies
	\[
	f(x)=0\quad\mbox{ and }\quad A\subset f^{-1}(1).
	\]
	Similarly, for every $A\in \mathcal{C}^{\downarrow}$ and $x\notin A$, there exists an increasing continuous function $f\colon E\to [0,1]$ such that
	$A\subset f^{-1}(0)$ and $f(x)=1$.
\end{assumption}

	\begin{remark}\label{rem:ass:top}
		If the preorder $\le$ is trivial, then the previous assumption corresponds to that of complete regularity.  
		Both metric spaces, and Hausdorff topological abelian groups are completely regular topological spaces and thus satisfy Assumption~\ref{ass:sep} for the trivial preorder. 
		Moreover, we show in Theorem~\ref{thm:separation} in Appendix A that every  preordered Hausdorff topological abelian group $(E,\le)$ as in Remark~\ref{rem:preorderedGroup}  satisfies Assumption~\ref{ass1} and Assuption~\ref{ass:sep}.      
	\end{remark}

Under Assumption \ref{ass:sep}, the following holds. 
\begin{lemma}\label{lem:inclusion}
	For any $O\in\mathcal{O}^{\uparrow}$ and $x\in O$, there exists $U\in\mathcal{U}_x$ such that ${\rm cl}({\uparrow}U)\subset O$.
\end{lemma}
\begin{proof}
	Fix $O\in\mathcal{O}^{\uparrow}$ and $x\in O$. Due to Lemma \ref{lem:basictop}, we have that $O^c\in \mathcal{C}^{\downarrow}$. 
	By Assumption \ref{ass:sep}, there exists an increasing continuous function $f\colon E\to [0,1]$ such that $O^c\subset f^{-1}(0)$ and $f(x)=1$. 
	Since $x\in \{f>1/2\}$, there exists $U\in\mathcal{U}_x$, which satisfies $U\subset\{f>1/2\}$. 
	Then, since $\{f>1/2\}$ is upwards closed, we have ${\uparrow}U\subset \{f>1/2\}$. Hence, 
	$x\in {\uparrow}U\subset {\rm cl}({\uparrow}U)\subset \{f\ge 1/2\}\subset O$.
	The proof is complete. 
\end{proof}
The previous result allows for a better description of the minimal rate function. 
\begin{lemma}\label{lem:ratefunction} 
	Let $J$ be a concentration function with minimal rate function $I_{\rm min}$. Then,  for every $x\in E$, 
	\[
	I_{\rm min}(x)=\sup_{f\in C_b^{\uparrow}}\{f(x)-\phi_J(f)\}=
	-\inf_{U\in \mathcal{U}_x} J_{{\uparrow}U}=-\inf_{U\in \mathcal{U}_x} J_{{\rm cl}({\uparrow}U)}.
	\]
\end{lemma}
\begin{proof}
	Let $x\in E$. By Lemma~\ref{lem:ratefunction1} and Lemma~\ref{lem:inclusion}, 
	\[
	\sup_{f\in C_b^{\uparrow}}\{f(x)-\phi_J(f)\}\le I_{\rm min}(x)=
	-\inf_{U\in \mathcal{U}_x} J_{{\uparrow}U}=-\inf_{U\in \mathcal{U}_x} J_{{\rm cl}({\uparrow}U)}.
	\]
	Recall that ${\rm cl}({\uparrow}U)\in  \mathcal{C}^\uparrow$ due to Lemma~\ref{lem:basictop}. Let $V\in\mathcal{U}_x$ and $r<0$.  It follows from Lemma~\ref{lem:basictop} that $({\uparrow}V)^c\in \mathcal{C}^{\downarrow}$. Hence, due to
	Assumption~\ref{ass:sep}, there exists an increasing continuous function $f_{U,r}\colon E\to [r,0]$ with $({\uparrow}U)^c\subset f^{-1}_{U,r}(r)$ and $f_{U,r}(x)=0$.  
	Then, we have
	\[
	\sup_{f\in C_b^{\uparrow}}\{f(x)-\phi_J(f)\}\ge -\phi_J(f_{U,r})\ge -\phi_J(r 1_{({\uparrow}U)^c}).
	\]
	Letting $r\to -\infty$ and using Proposition \ref{prop:concentration}, we see that $\sup_{f\in C_b^{\uparrow}}\{f(x)-\phi_J(f)\}\ge -J_{{\uparrow}U}$.
	If we take the supremum over all $U\in\mathcal{U}_x$, we conclude
	\[
	-\inf_{U\in \mathcal{U}_x} J_{{\uparrow}U} \le \sup_{f\in C_b^{\uparrow}}\{f(x)-\phi_J(f)\}.\qedhere
	\] 
\end{proof}
Motivated by the theory of large deviations and in accordance with the inequalities~\eqref{mot:bounds}, we introduce the following concepts. 
\begin{definition} 
	Let $I\colon E\to[0,\infty]$ be a function. A concentration $J$ is said to satisfy the \emph{monotone large deviation principle} (mLDP) with rate function $I$ if 
	\begin{equation}\label{eq:LDP}
	-\inf_{x\in O}I(x) \le J_{O}\quad \mbox{and}\quad
	J_C\le -\inf_{x\in C}I(x)
	\end{equation}
	for every $O\in\mathcal{O}^\uparrow$ and all $C\in\mathcal{C}^\uparrow$. Moreover, a concentration $J$ is said to satisfy the \emph{monotone Laplace principle} (mLP) with rate function $I$ if the maxitive integral $\phi_J$ has the representation  
	\begin{equation}
		\label{eq:LP}
		\phi_J(f)=\sup_{x\in E}\{f(x)-I(x)\}\quad\mbox{ for all }f\in C^\uparrow_b.
	\end{equation}
\end{definition}	
The mLDP is equivalent to the bounds~\eqref{eq:LDPL} and~\eqref{eq:LDPU}. On the other hand, while the bounds~\eqref{eq:L1} and \eqref{eq:U1} imply the mLP, the converse assertion is not necessarily true.
For a concentration $J$ which satisfies either the mLDP or mLP with rate function $I$, it follows from $J_E=0$ or equivalently $\phi_J(0)=0$, that 
$I$ is proper, i.e., $I(x)\in[0,\infty)$ for some $x\in E$.  We do not require that rate functions are lower semicontinuous.\footnote{In the theory of large deviations, it is typically assumed that a rate function is lower semicontinuous and proper; see e.g. \cite{dembo}.} However,
we show that the mLDP uniquely determines the rate function within the class of increasing lower semicontinuous functions. 
For a function $f\colon E\to[-\infty,\infty]$, we define its \emph{increasing lower semicontinuous envelope}  $f^\uparrow\colon E\to[-\infty,\infty]$ by
\[
f^{\uparrow}(x):=\sup \big\{g(x)\colon g\in L^\uparrow,\, g\le f\big\}. 
\]
Directly from the definition, we see that $f^{\uparrow}$ is the greatest increasing lower semicontinuous function $g\colon E\to[-\infty,\infty]$  which satisfies that $g\le f$.

\begin{proposition}\label{prop:uniquenessrate}
	Let $J$ be a concentration and $I\colon E\to[0,\infty]$ be a function.  
	\begin{itemize}
		\item[(i)]  If $J_{C}\le -\inf_{x\in C} {I}(x)$ for every $C\in\mathcal{C}^\uparrow$, then  $ I^{\uparrow}\le I_{\rm min}$.
		\item[(ii)] If $-\inf_{x\in O}{I}(x)\le J_{O}$ for every $O\in\mathcal{O}^\uparrow$, then $I_{\rm min}\le I^{\uparrow}$. 
	\end{itemize}
	In particular, if $J$ satisfies the mLDP with rate function $I$ being increasing and lower semicontinuous, then $I=I_{\rm min}$.    
\end{proposition}
\begin{proof}
	First, we assume that $J_{C}\le -\inf_{x\in C} I(x)$ for all $C\in\mathcal{C}^\uparrow$.  
	Since $I^\uparrow\le I$, 
	\[J_C\le  -\inf_{x\in C}I(x)\le -\inf_{x\in C}I^\uparrow(x)\quad\mbox{ for all }C\in\mathcal{C}^\uparrow.\]
	Let $x\in E$.  We apply the previous inequality and the fact that $I^\uparrow$ is lower semicontinuous and increasing to show that
	\[I^\uparrow(x)=\underset{U\in\mathcal{U}_x}\sup\underset{y\in U}\inf I^\uparrow(y)
	=\underset{U\in\mathcal{U}_x}\sup\underset{y\in {\uparrow}U}\inf I^\uparrow(y)
	=\underset{U\in\mathcal{U}_x}\sup \underset{y\in {\rm cl}({\uparrow}U)}\inf I^\uparrow(y)\le 
	-\underset{U\in\mathcal{U}_x}\inf J_{{\rm cl}({\uparrow}U)}=I_{\rm min}(x),\]
	where in the first equality we used that $I^\uparrow$ is lower semicontinuous,
	in the second equality that $I^\uparrow$ is increasing, the third equality is a consequence of Lemma~\ref{lem:inclusion},
	and the last equality follows from Lemma~\ref{lem:ratefunction}. 
	
	Second, we assume that $-\inf_{x\in O}{I}(x)\le J_{O}$ for every $O\in\mathcal{O}^\uparrow$. Then, it follows from
	Lemma~\ref{lem:ratefunction} that for every $x\in E$, 
	\[
	I_{\rm min}(x)=-\underset{U\in \mathcal{U}_x}\inf J_{{\uparrow}U}=\underset{U\in \mathcal{U}_x}\sup (-J_{{\uparrow}U})\le \underset{U\in \mathcal{U}_x}\sup\inf_{y\in {\uparrow}U}I(y)\le I(x),
	\] 
	where we have used that ${\uparrow}U\in \mathcal{O}^\uparrow$. 
	This shows $\I\le I$. Moreover, since $\I$ is increasing and lower semicontinuous, it follows that $\I\le I^\uparrow$. 
	
	In particular, if $J$ satisfies the mLDP with an increasing lower semicontinuous rate function $I$, then $I=I^\uparrow$, so that conditions (i) and (ii) imply that 
	$\I\le I\le \I$. 
\end{proof}	
It follows from Theorem \ref{thm:mVaradhan1} that the mLDP implies the mLP.
The converse implication is more involved. As shown in the following main result, if the space is monotonically normal or the rate function has downwards compactly generated sublevel sets, then the mLP implies the mLDP.
\begin{theorem}\label{thm:mLPimpliesLDP}
	Let $I\colon E\to [0,\infty]$ be an increasing  lower semicontinuous function, and suppose that one of the following conditions is satisfied:
	\begin{enumerate}
		\item $E$ is monotonically normal,\footnote{This notion is due to Nachbin~\cite{nachbin1965topology} where the term `normally ordered space' is used in the context of an ordered set. 
			The term `monotonically normal' is also used in the literature for preordered sets; cf.~e.g.,~\cite{bonsall,semadeni1968preordered}.} if for every $A\in\mathcal{C}^{\downarrow}$ and $B\in\mathcal{C}^{\uparrow}$ with $A\cap B=\emptyset$ there exist $U\in \mathcal{O}^\downarrow$ and $V\in \mathcal{O}^\uparrow$ with $U\cap V=\emptyset$ such that $A\subset U$ and $B\subset V$.
		\item $I$ has downwards compactly generated sublevel sets, i.e., if for every $\alpha\in\R$ there exists a compact set $K_\alpha$ such that 
		$\{I\le \alpha\}={\downarrow}K_\alpha$. 
	\end{enumerate}
	Then, $J$ satisfies the mLP with rate function $I$ if and only if 
	$J$ satisfies the mLDP with rate function $I$. 
	In that case, $I=I_{\rm min}$. 
\end{theorem}
\begin{proof}  
	If $J$ satisfies the mLDP with increasing rate function $I$, then it follows from Theorem \ref{thm:mVaradhan1} that $J$ satisfies the mLP with rate function $I$.
	Moreover, Proposition~\ref{prop:uniquenessrate} ensures that  $I=I_{\rm min}$. 
	
	Conversely, suppose that the mLP with rate function $I$ holds. Fix $A\in\mathcal{C}^\uparrow$.   
	We first show that $J_A\le-\inf_{x\in A} I(x)$.  
	If $\inf_{x\in A} I(x)=0$, then the assertion trivially holds as $J_A\le 0$.  
	Thus, assume that $\inf_{x\in A} I(x)>\delta$ for some $\delta>0$ small enough.  Define \[I^\delta(x):=(I(x)-\delta)\wedge \delta^{-1}\quad\mbox{for }x\in E.\]  
	Next, we show that
	\begin{equation}
		\label{eq:infalpha}
		J_{A}\le -\alpha,
	\end{equation}
	where $\alpha:=\inf_{x\in A}I^\delta(x)\in(0,\infty)$. 
	Since $-\inf_{x\in E} I(x)=\phi_J(0)=0$, it follows 
	that $\{I\le \alpha\}$ is non-empty. 
	Moreover, $\{I\le \alpha\}\in \mathcal{C}^\downarrow$ as $I$ is increasing and lower semicontinuous. 
	In addition, it holds $\{I\le\alpha\}\cap A=\emptyset$. 
	
	Suppose first that $E$ is monotonically normal.  It follows from \cite[Theorem 2]{nachbin1965topology} that for each $m\in\mathbb{N}$, there exists an increasing continuous function $h_m\colon E\to [-m,0]$ such that $\{I\le \alpha\}\subset h_m^{-1}(-m)$ and $A\subset h_m^{-1}(0)$.\footnote{ \cite[Theorem 2]{nachbin1965topology} is formulated for ordered topological spaces rather than preordered topological spaces. 
		However, since in the proof the antisymmetry is not used, the result applies also to preordered topological spaces.} 
	For each $m\in\N$, 
	\[
	J_{A}=\inf_{r\in\mathbb{R}}\phi_J(r 1_{A^c})\le \phi_J(h_m)=-\inf_{x\in S}\{-h_m(x)+I(x)\}.
	\]
	Since $-h_m(x)+I(x)\ge m$ if $x\in \{I\le \alpha\}$, and $-h_m(x)+I(x)\ge \alpha$ if $x\notin \{I\le \alpha\}$, by choosing  $m\ge \alpha$, we obtain inequality~\eqref{eq:infalpha}.
	
	Suppose now that $I$ has downwards compactly generated sublevel sets. 
	In that case, $\{I\le \alpha\}={\downarrow}K_\alpha$ for a compact set $K_\alpha\subset E$. 
	Since $A\cap K_\alpha=\emptyset$, by Assumption \ref{ass:sep}, for every $y\in K_\alpha$ there exists an increasing continuous function $f_y\colon E\to [-1,0]$ such that $f_y(y)=-1$ and $A\subset f_y^{-1}(0)$. 
	The neighborhoods $V_y:=\{f_y<-\tfrac{1}{2}\}$ cover the compact set $K_\alpha$. 
	Hence, we can find $y_1,\ldots,y_N\in K_\alpha$ such that $K_\alpha\subset \bigcup_{1\le i\le N} V_{y_i}$. 
	For each $m\in\N$, we define $h_m:=\wedge_{1\le i\le N} 2 m f_{y_i}$.  
	Then, $h_m$ is increasing, continuous and bounded, $A\subset h_m^{-1}(0)$ and $h_m(x)\le -m$ for all $x\in K_\alpha$.    
	Moreover, since $h_m$ is increasing, it follows that $h_m(x)\le -m$ for all $x\in \{I\le \alpha\}={\downarrow}K_\alpha$. 
	Finally, with the same arguments as in the previous case,  we obtain inequality~\eqref{eq:infalpha}.  
	
	In both cases, inequality~\eqref{eq:infalpha} holds for all $\delta>0$ small enough, so that 
	\[J_A\le-\lim_{\delta\downarrow 0}\inf_{x\in A} I^\delta(x)
	=-\lim_{\delta\downarrow 0}\Big((\inf_{x\in A} I(x) - \delta)\wedge \delta^{-1}\Big)
	=-\inf_{x\in A} I(x).\] 
	Hence, $J$ satisfies the upper bound of the mLDP with increasing rate function $I$.
	By Proposition \ref{prop:uniquenessrate}, we obtain $I_{\rm min}\ge I$. On the other hand,
	it follows from the representation \eqref{eq:LP} that $I_{\rm min}\le I$, and therefore $I=I_{\rm min}$. 
	Finally, it follows from Theorem~\ref{thm:mVaradhan1} that $J$ satisfies the lower bound of the mLDP
	with rate function $I=I_{\rm min}$. 
\end{proof}

	\begin{remark}
		If the preorder $\le$ is trivial, then the notion of monotonical normality corresponds to normality. 
		Examples of normal spaces include metrizable spaces, and regular Lindel\"{o}f spaces; see e.g. \cite{kelley}.   
		  
		In case that $\le$ is a closed preorder\footnote{I.e., $\{(x,y)\colon x\le y\}$ is closed.},  it is shown in \cite{minguzzi} that $E$ is monotonically normal whenever $E$ is second countable and locally compact. This is the case, for example, when $E$ is a (second countable)  topological manifold equipped with a closed preorder. 
		Further examples of  monotonically normal spaces can be found in \cite{bosi,minguzzi}.   
		 
		In case that $E=\mathbb{R}$ is the real line, then every increasing lower semicontinuous function $I\colon E\to[0,\infty)$ satisfies that $\{I\le \alpha\}=(-\infty,I(\alpha)]={\downarrow}\{I(\alpha)\}$, and has therefore compactly generated sublevel sets. 
	\end{remark}

\section{Convex rate functions}\label{sec:convRateFunc}
The explicit determination of the rate function is generally a difficult task.
However, in the particular case where the rate function is convex, one can rely on convex duality arguments.
Throughout this section, let $E$ be a  locally convex Hausdorff topological real vector space. Moreover, let $E_+\subset E$ be a convex cone, i.e., 
$E_+ + E_+ \subset E_+$, $\lambda E_+ \subset E_+$ for all $\lambda>0$ and $0\in E_+$.
We endow $E$ with the preorder induced by $E_+$, i.e., $x\le y$ if and only if $y-x\in E_+$. 
Then, the assumptions of Section~\ref{sec:monLDP} are satisfied as outlined in Remark~\ref{rem:ass:top}.  
Let $\mathcal{U}$ be a base for the topology of $E$, and define $\mathcal{U}_x:=\{U\in\mathcal{U}\colon x\in U\}$ for all $x\in E$. 
We first provide a condition which ensures that the minimal rate function is convex. 

\begin{lemma}\label{lem:convexRate}
	Let $J$ be a concentration such that 
	\[
	J_{\tfrac{1}{2}{\uparrow}U + \tfrac{1}{2}{\uparrow}V}\ge \tfrac{1}{2}J_{{\uparrow}U} + \tfrac{1}{2}J_{{\uparrow}V}\quad\mbox{ for all }U,V\in\mathcal{U}.
	\]
	Then, the minimal rate function $I_{\rm min}\colon E\to [0,\infty]$ is convex.
\end{lemma}
\begin{proof}
	Fix $x,y\in E$ and $\varepsilon>0$. 
	Set $z:=\frac{1}{2}x+\frac{1}{2}y$. 
	Due to Lemma~\ref{lem:ratefunction}, there exists $W\in\mathcal{U}_z$ such that
	\[
	-J_{{\uparrow}W} \ge (\I(z)-\varepsilon)\wedge 	\varepsilon^{-1}.
	\]
	Since $O:=\{(\tilde{x},\tilde{y})\in E\times E\colon \frac{1}{2}\tilde{x}+ \frac{1}{2}\tilde{y}\in W\}$ is open, there exist  $U\in\mathcal{U}_x$ 
	and $V\in\mathcal{U}_y$ such that $\frac{1}{2} U+\frac{1}{2} V\subset W$. 
	Then, 
	\[
	-[(\I(z)-\varepsilon)\wedge\varepsilon^{-1}]\ge J_{{\uparrow}W}\ge J_{\frac{1}{2}{\uparrow}U+\frac{1}{2}{\uparrow}V}\ge \frac{1}{2}J_{{\uparrow}U}+\frac{1}{2}J_{{\uparrow}V}.
	\]
	Taking the infimum over all $U\in\mathcal{U}_x$ and then over all $V\in \mathcal{U}_y$, 
	\[
	(\I(z)-\varepsilon)\wedge\varepsilon^{-1}\le \frac{1}{2}\I(x)+\frac{1}{2}\I(y).
	\]
	Therefore, since $\varepsilon>0$ was arbitrary, we obtain 
	\[
	\I(z)\le \frac{1}{2}\I(x)+\frac{1}{2}\I(y).
	\]
	Now, consider the set the set $\mathcal{D}:=\{k2^{-n}\colon n,k\in\mathbb{N},\: k\le 2^n\}$ of all  dyadic rational numbers in the interval $[0,1]$. 
	By recursion, it can be shown that 
	\[
	\I(z)\le \lambda\I(x)+(1-\lambda)\I(y),
	\]
	for every $\lambda\in \mathcal{D}$. As a consequence, since the map $[0,1]\to[0,\infty]$, $\lambda\mapsto \I(\lambda x+(1-\lambda) y)$ 
	is lower semicontinuous and the set $\mathcal{D}$  is dense  in $[0,1]$, the previous inequality is valid for all~$\lambda\in[0,1]$.  
\end{proof}
Let $E^\ast$ be the topological dual space of $E$.  We denote by $E^{\ast}_+$ the set of all $\mu\in E^\ast$ which are positive, i.e., $\mu(x)\ge 0$ for all $x\in E_+$. 
The convex conjugate of the minimal rate function $\I$ is defined by $ \I^\ast(\mu):=\sup_{x\in E}\{\mu(x)-\I(x)\}$ for all $\mu\in E^\ast$.

\begin{proposition}\label{prop:conI}
		Let $J$ be a concentration. Suppose that the minimal rate function $\I$ is convex and $ \I^\ast(\mu)\ge \phi_J(\mu)$ for all $\mu\in E^\ast_+$. Then,
	\begin{equation}\label{eq:repRate}
		\I(x)=\underset{\mu \in E_+^\ast}\sup\{\mu(x)-\phi_J(\mu)\}\quad\mbox{for all }x\in E.
	\end{equation}   	
\end{proposition}	
\begin{proof}
	The minimal rate function $I_{\rm min}$ is lower semicontinuous and increasing.
	By the Fenchel-Moreau theorem \cite[Theorem A.62]{follmer2016stochastic}, 
	\begin{align*}
		\I(x)&=\sup_{\mu\in E^\ast} \{\mu(x)-\I^\ast(\mu)\}\\
		&=\sup_{\mu\in E^\ast_+} \{\mu(x)-\I^\ast(\mu)\}
		\quad\mbox{for all }x\in E.
	\end{align*}
	The second equality holds because $\I^\ast(\mu)=\infty$ whenever $\mu\in E^\ast\setminus E^\ast_+$. 
	In fact, for $\mu\in E^\ast\setminus E^\ast_+$ there exists $y\le 0$ with $\mu(y)>0$, and therefore
	\begin{align*}
		\I^\ast(\mu)&\ge \underset{\lambda>0}\sup\{\mu(x+\lambda y)-\I(x+\lambda y)\}\\
		&\ge \underset{\lambda>0}\sup\{\mu(x)+\lambda\mu(y)-\I(x)\}=\infty. 
	\end{align*} 
	Finally, using the definition of $I_{\min}$ and the inequality  $\I^\ast(\mu)\ge \phi_{J}(\mu)$ for all $\mu\in E^\ast_+$,
	\[
	I_{\min}(x)=\sup_{f\in {L}^\uparrow}\{ f(x)-\phi_J(f)\}\ge \sup_{\mu\in E^\ast_+}\{ \mu(x)-\phi_J(\mu)\}\ge \I(x).\qedhere
	\]
\end{proof}	
The hypothesis of Proposition \ref{prop:conI} can be verified in some important situations; e.g., for the asymptotic concentration of sample means of i.i.d.~sequences, see Subsection~\ref{subsec:iid}. In particular, it is satisfied under the mLDP, which leads to the following  result.

\begin{corollary}
	Let $J$ be a concentration which satisfies the mLDP with a rate function $I\colon E\to[0,\infty]$ which is convex, increasing and lower semicontinuous.
	Then, \[I(x)=\underset{\mu \in E_+^\ast}\sup\{\mu(x)-\phi_J(\mu)\}\quad \mbox{for all }x\in E.\]
\end{corollary}
\begin{proof}
	By Proposition~\ref{prop:uniquenessrate},  it holds $I=\I$. Moreover, it follows from Theorem~\ref{thm:mVaradhan1} that $\phi_J(\mu)=\I^\ast(\mu)$ for all $\mu\in E^\ast_+$.
	Consequently,  since $I=\I$ is convex, the claim follows from Proposition \ref{prop:conI}.
\end{proof}

\section{Examples}\label{sec:examples}
We illustrate the theoretical results with two examples.
First, we study the asymptotic behavior of a sequence $(\mu_n)_{n\in\N}$ of capacities by considering the weakly maxitive concentration 
$J_A:=\limsup_{n\to\infty}\tfrac{1}{n}\log\mu_n(A)$. Large deviations bounds for sequences of capacities on $\R^d$ were recently considered in \cite{chen2016large,tan2020large}.
Second, we focus on the role of the partial order by elaborating a 
monotone version of Cram\'{e}r's theorem for which the rate function can be determined explicitly. 

\subsection{Asymptotic concentration of capacities}\label{sec:acc}
Let $(E,\le)$ be a topological preordered space which satisfies Assumption \ref{ass1}.  
We denote by $\overline{B}$ the set of all Borel measurable functions $f\colon E\to[-\infty,\infty)$ which are bounded from above.  
In the following, let  $(\mathcal{E}_n)_{n\in\N}$ be a sequence of sublinear expectations on $\overline{B}$, i.e., for each $n\in\N$, the
functional $\mathcal{E}_n\colon \overline{B}\to [-\infty,\infty)$ satisfies
\begin{itemize}
	\item[(i)] $\mathcal{E}_n(f)\le \mathcal{E}_n(g)$ whenever $f\le g$,
	\item[(ii)] $\mathcal{E}_n(f+g)\le \mathcal{E}_n(f)+\mathcal{E}_n(g)$, 
	\item[(iii)] $\mathcal{E}_n(f+c)=\mathcal{E}_n(f)+c$ for all $c\in\R$,
	\item[(iv)] $\mathcal{E}_n(\lambda f)=\lambda\mathcal{E}_n(f)$ for all $\lambda\in[0,\infty)$.
\end{itemize}
A functional which satisfies the properties (i)-(iv) is also called upper expectation in robust statistics \cite{huber}, coherent risk measure in mathematical finance \cite{artzner1999coherent}, or upper coherent prevision in the theory of imprecise probabilities \cite{walley1991statistical}. 
We remark that the theory of upper previsions does not make any measurability assumption, and in that context it is known that the conditions (i) and (iii) follow from the conditions (ii) and (iv) together with $\mathcal{E}_n(f)\le\sup f$; see, e.g.,~\cite{miranda2,walley1991statistical}.  
\begin{example}
Let $(X_n)_{n\in\N}$ be a sequence of $E$-valued random variables defined on a probability space $(\Omega,\mathcal{F},\P)$. 
Consider a non-empty set $\mathcal{P}$ of probability measures on $\mathcal{F}$. 
Then, for each $n$, the upper prevision $\mathcal{E}_n(f)=\sup_{\mathbb{P}\in\mathcal{P}}\mathbb{E}_{\mathbb{P}}[f(X_n)]$ satisfies the properties (i)--(iv) above.
\end{example}

We consider the set function
\[
	J\colon \mathcal{OC}^{\uparrow}\to [-\infty,0],\quad J_A:=\underset{n\to\infty}{\limsup}\tfrac{1}{n}\log\mu_n(A),
\]
where $\mu_n$ denotes the corresponding capacity of $\mathcal{E}_n$, which is defined by  $\mu_n(A):=\mathcal{E}_n(1_A)$ for all Borel sets $A\subset E$. 
Straightforward verification shows that  each $\mu_n$ satisfies $\mu_n(\emptyset)=0$, $\mu_n(E)=1$, $\mu_n(A)\le \mu_n(B)$ whenever $A\subset B$, and  $\mu_n(A)\le\mu_n(B)+\mu_n(C)$ whenever $A\subset B\cup C$. Moreover, $J$ is a concentration which turns out to be weakly maxitive according to the principle of the largest term.\footnote{The \emph{principle of the largest term} is a result which is often used in  the theory of large deviations; see, e.g.,~\cite[Proposition 12.3]{petit} and \cite[Lemma 1.2.15]{dembo}. 
	Namely, if $(a_n^1)_{n\in\N},\ldots,(a_n^N)_{n\in\N}$ are $[0,\infty]$-valued sequences, then $\limsup_{n\to\infty} \tfrac{1}{n}\log\sum_{i=1}^N a_n^i
	=\vee_{i=1}^N\limsup_{n\to\infty} \tfrac{1}{n}\log a_n^i.$}
\begin{lemma}\label{lem:seqCap}
	The concentration $J$ is weakly maxitive.
\end{lemma}
\begin{proof}
	Let $C\in \mathcal{C}^{\uparrow}$ and $O_1,O_2,\ldots,O_N\in \mathcal{O}^{\uparrow}$ such that
	$C\subset \cup_{i=1}^N O_1$, and therefore $\mu_n(C)\le \sum_{i=1}^N\mu_n(O_i)$. Applying the principle of the largest term,
	\[
		J_C \le \limsup_{n\to\infty}\tfrac{1}{n}\log \sum_{i=1}^N\mu_{n}(O_i)\le \vee_{i=1}^N\limsup_{n\to\infty}\tfrac{1}{n}\log \mu_{n}(O_i)= \vee_{i=1}^N J_{O_i}.
		\qedhere\]
\end{proof}	
In particular, the results of Section~\ref{sec:maxInt} are applicable. Due to Lemma~\ref{lem:ratefunction1}, the minimal rate function is given by 
\[ I_{\rm min}(x)=-\sup_{U\in\mathcal{U}_x} \limsup_{n\to\infty}\tfrac{1}{n}\log \mu_n\left({\uparrow}U\right)\quad\mbox{for all }x\in E.\] 
As an application of  Theorem~\ref{thm:mVaradhan1}, the discussion thereafter, and Proposition~\ref{prop:LDPupper}, we obtain 
\begin{equation}\label{eq:ex:lowerupper}
	-\underset{x\in O}\inf I_{\rm min}(x)\le J_O \quad \mbox{and} \quad 	J_C\le -\underset{x\in C}\inf I_{\rm min}(x)
\end{equation}
for all $O\in\mathcal{O}^\uparrow$ and $C\in\mathcal{C}_c^\uparrow$, and consequently the respective bounds  in Theorem~\ref{thm:mVaradhan1} for the maxitive integral  $\phi_{J}$.
Building on Theorem~\ref{thm:maxitiveRep} and Remark~\ref{rem:events}, the maxitive integral  $\phi_{J}$ has the following representation in terms of a sequence of Choquet integrals.
\begin{proposition}\label{eq:seqCap}
	For every $f\in \overline{L}^\uparrow\cap \overline{U}^\uparrow$, 
	\begin{equation}
		\label{eq:maxIntPi}
		\phi_{J}(f)=\limsup_{n\to\infty}\tfrac{1}{n}\log \mathcal{E}_n(\exp(nf))\\
		=\limsup_{n\to\infty}\tfrac{1}{n}\log \int_0^\infty \mu_n\left(\exp(n f)>x\right){\rm d}x.
	\end{equation}
\end{proposition}
\begin{proof}
	Define $\psi\colon \overline{LU}^\uparrow\to [-\infty,\infty)$ as the right hand side of~\eqref{eq:maxIntPi}. 
	Inspection shows that $\psi$ satisfies properties (i)-(iii) in Theorem~\ref{thm:maxitiveRep}. 
	Next, we show that $\psi$ is weakly maxitive.  To that end, let $f\in \overline{U}^{\uparrow}$, $g_1,g_2,\ldots,g_N\in \overline{L}^{\uparrow}$, and $f\le \vee_{i=1}^N g_1$.    
	Then, by the principle of the largest term, 
	\begin{align*}
		\psi(f)	
		&\le \limsup_{n\to\infty}\tfrac{1}{n}\log \int_0^\infty \mu_n\left(\exp\left(n (\vee_{i=1}^N g_i)\right)> x\right){\rm d}x\\
		&= \limsup_{n\to\infty}\tfrac{1}{n}\log \int_0^\infty \mu_n\left(\cup_{i=1}^N \{\exp(n g_i)> x\}\right){\rm d}x\\
		&\le \limsup_{n\to\infty}\tfrac{1}{n}\log \sum_{i=1}^N \int_0^\infty \mu_n\left(\exp(n g_i)> x\right){\rm d}x\\
		&\le \vee_{i=1}^N\limsup_{n\to\infty}\tfrac{1}{n}\log \left(\int_0^\infty \mu_n\left(\exp(n g_i)> x\right){\rm d}x\right)\\
		&= \vee_{i=1}^N \psi(g_i),
	\end{align*}
which shows that $\psi$ is weakly maxitive. Moreover,  for $A\in \mathcal{OC}^{\uparrow}$, 
	\begin{align*}
		\psi(-\infty 1_{A^c})&=\underset{n\to\infty}{\limsup}\tfrac{1}{n}\log \int_0^\infty \mu_n\left(\exp(n (-\infty)1_{A^c})> x\right){\rm d}x\\
		&=\limsup_{n\to\infty}\tfrac{1}{n}\log \int_0^\infty \mu_n\left(1_A> x\right){\rm d}x\\
		&=\limsup_{n\to\infty}\tfrac{1}{n}\log \int_0^1 \mu_n\left(A\right){\rm d}x\\
		&=\limsup_{n\to\infty}\tfrac{1}{n}\log \mu_n\left(A\right)\\
		&=J_A.
	\end{align*} 
Hence, $J^\psi=J$ and we can apply Theorem~\ref{thm:maxitiveRep} to conclude $\psi=\phi_{J}$. 

In a second step, we define 
\[
\tilde\psi\colon \overline{LU}^\uparrow\to [-\infty,\infty), \quad \tilde\psi(f):=\limsup_{n\to\infty}\tfrac{1}{n}\log \mathcal{E}_n\big(\exp(nf)\big).
\]
Direct verification shows that $\tilde\psi$ satisfies properties (i)-(iii) in Theorem~\ref{thm:maxitiveRep}. 
Moreover, it is weakly maxitive. Indeed, let $f\in \overline{U}^{\uparrow}$ and $g_1,g_2,\ldots,g_N\in \overline{L}^{\uparrow}$ such that $f\le \vee_{i=1}^N g_1$.    
Again, by the principle of the largest term, 
	\begin{align*}
		{\tilde\psi}(f)&\le \underset{n\to\infty}\limsup \frac{1}{n}\log\mathcal{E}_n\left(e^{n (\vee_{i=1}^N g_i)}\right)\\
		&\le \underset{n\to\infty}\limsup\frac{1}{n}\log \sum_{i=1}^N\mathcal{E}_n\left(  e^{n g_i}\right) \\
		&\le \vee_{i=1}^N \underset{n\to\infty}\limsup\frac{1}{n}\log \mathcal{E}_n\left(  e^{n g_i}\right) \\
		&=\vee_{i=1}^N{\tilde\psi}(g_i).
	\end{align*}
	Since $J_A=\tilde\psi(-\infty 1_{A^c})$ for all $A\in \mathcal{OC}^{\uparrow}$, it follows from Theorem~\ref{thm:maxitiveRep} that $\tilde\psi=\phi_{J}$. 
	Together with the first part, we obtain equality~\eqref{eq:seqCap}.
\end{proof}
	In contrast to linear expectations in standard probability theory, a sublinear expectation $\mathcal{E}$ is in general not determined by its associated capacity $\mu(A)=\mathcal{E}(1_A)$; see, e.g.,~\cite{peng2019nonlinear}. In other words, the sublinear expectation $\mathcal{E}$ may contain more information than the capacity $\mu$. 
    However,  Proposition~\ref{eq:seqCap}  implies that the asymptotic entropic version $\limsup_{n\to\infty}\tfrac{1}{n}\log \mathcal{E}_n(\exp(nf))$
is fully specified through the sequence of capacities $(\mu_n)_{n\in\N}$ by means of the right hand side of equation~\eqref{eq:maxIntPi}.  

Under the assumptions of Theorem~\ref{thm:mLPimpliesLDP}, it follows that $\phi_J$ satisfies the mLP with rate function $I$ if and only if $J$ satisfies the mLDP with rate function $I$. Under a slightly stronger version of the mLDP, the limit superior in \eqref{eq:maxIntPi} can even be replaced by a limit. More precisely, we obtain the equivalence between the following versions of the classical large deviation principle and the Laplace principle.
\begin{corollary}
	Suppose that $(E,\le)$ satisfies  Assumption \ref{ass1} and  Assumption \ref{ass:sep}. 
	Let $I\colon E\to [0,\infty]$ be an increasing  lower semicontinuous function.	If $(\mu_n)_{n\in\mathbb{N}}$ satisfies
	\begin{equation}\label{muLDP}
		 -\underset{x\in O}\inf I(x)\le \underset{n\to\infty}{\liminf}\tfrac{1}{n}\log\mu_n(O)\qquad\mbox{and}\qquad	\underset{n\to\infty}{\limsup}\tfrac{1}{n}\log\mu_n(C)\le -\underset{x\in C}\inf I(x)	
	\end{equation}	
	for all $O\in\mathcal{O}^\uparrow$ and $C\in\mathcal{C}^\uparrow$, then $(\mathcal{E}_n)_{n\in\mathbb{N}}$ satisfies 
		\begin{equation}\label{muLP}
	\underset{n\to\infty}\lim\tfrac{1}{n}\log \mathcal{E}_n(\exp(nf))
	=\underset{x\in E}\sup\{f(x)-I(x)\}
	\end{equation}
	for all $f\in C^\uparrow_b$. The converse assertion holds true if $E$ is monotonically normal or $I$ has compactly generated sublevel sets. 
\end{corollary}
\begin{proof}
Let $f\in C_b^\uparrow$. It follows from Theorem \ref{thm:mVaradhan1} that
\[
\phi_J(f)=\underset{n\to\infty}\limsup\tfrac{1}{n}\log \mathcal{E}_n(\exp(nf))
=\underset{x\in E}\sup\{f(x)-I(x)\}.
\] 
To show that the previous limit superior is a limit, we define \[\underline{\phi}(f):=\underset{n\to\infty}\liminf\tfrac{1}{n}\log \mathcal{E}_n(\exp(nf)).\]
Fix $\varepsilon>0$ and $x\in E$. 
Since $f$ is upper semicontinuous and increasing, there exists $U\in \mathcal{O}^\uparrow$ with $x\in U$ such that $\inf_{y\in U}f(y)\ge f(x)-\varepsilon$. 
Then,
\begin{align*}
\underline{\phi}(f)&\ge \underline{\phi}(f 1_U - \infty 1_{U^c})\\&=f(x)-\varepsilon +  \underline{\phi}( - \infty 1_{U^c}) \\
&=f(x)-\varepsilon + \underset{n\to\infty}{\liminf}\tfrac{1}{n}\log\mu_n(U)\\
&\ge f(x)-\varepsilon -\underset{y\in U} \inf I(y)\\
&\ge f(x)-\varepsilon - I(x).
\end{align*}
Letting $\varepsilon \downarrow 0$ and taking the supremum over all $x\in E$ yields 
\[
\underset{x\in E}\sup\{f(x)-I(x)\}\le \underline{\phi}(f)\le \phi_J(f)=\underset{x\in E}\sup\{f(x)-I(x)\}.
\]
As for the converse assertion, suppose that $E$ is monotonically normal or $I$ has compactly generated sublevel sets.
By Theorem \ref{thm:mLPimpliesLDP}, 
\[
\underset{n\to\infty}{\limsup}\tfrac{1}{n}\log\mu_n(C)=J_C\le -\underset{x\in C}\inf I(x)\quad\mbox{ for all }C\in \mathcal{C}^\uparrow.
\]
To show the lower bound, fix $O\in\mathcal{O}^\uparrow$. Let $x\in O$ and $m\in\mathbb{N}$. By Lemma~\ref{lem:basictop},
it holds $O^c\in \mathcal{C}^\downarrow$. Hence, due to Assumption~\ref{ass:sep}, there exists $f_m\in C_b^\uparrow$ such that $f_m(x)=0$, $O^c\subset f^{-1}_m(-m)$, and $-m\le f(y)\le 0$ for all $y\in E$.  We obtain
\begin{align*}
\underset{n\to\infty}{\liminf}\tfrac{1}{n}\log\mu_n(O)\vee (-m)
&=\underset{n\to\infty}\liminf\tfrac{1}{n}\log\mathcal{E}\big(\exp(n(-\infty 1_{O^c}))\big)\vee (-m)\\
&\ge \underset{n\to\infty}\liminf\tfrac{1}{n}\log\mathcal{E}\big(\exp(n((-\infty 1_{O^c})\vee (-m))\big)\\
&\ge \underline{\phi}(f_m)=\phi_J(f_m)\ge f_m(x)-I(x)=-I(x).
\end{align*}
By letting $m\to\infty$ and taking the supremum over all $x\in O$, 
\[
\underset{n\to\infty}{\liminf}\tfrac{1}{n}\log\mu_n(O)\ge -\underset{x\in O}\inf I(x).\qedhere
\] 
\end{proof}

\subsection{Sample means of i.i.d.~sequences}\label{subsec:iid}
Let $E$ be a locally convex Hausdorff topological real vector space endowed with a preorder induced by a closed convex cone $E_+\subset E$. Then, the assumptions of Section~\ref{sec:monLDP} are satisfied as outlined in Remark~\ref{rem:ass:top}.  

For each $n\in\N$, let $X_n:=\tfrac{1}{n}\sum_{i=1}^n\xi_i$ be the sample mean of an i.i.d. sequence $(\xi_n)_{n\in\N}$ of $E$-valued random variables defined on a probability space $(\Omega,\mathcal{F},\mathbb{P})$. 
For $A\in\mathcal{OC}^\uparrow$, we define the concentration $J_A=\limsup_{n\to\infty} \tfrac{1}{n}\log\mathbb{P}(X_n\in A)$ with minimal rate function ${\I}$.  
Due to Proposition \ref{eq:seqCap}, the concentration ${J}$ is weakly maxitive and its maxitive integral admits the representation
\begin{equation}\label{eq:repCramer}
	\phi_{J}(f)=\limsup_{n\to\infty}\tfrac{1}{n}\log\mathbb{E}_{\mathbb{P}}[\exp(n f(X_n))]\quad\mbox{for all }f\in \overline{L}^\uparrow\cap \overline{U}^\uparrow.
\end{equation}
\begin{lemma}\label{lem:limExists}
	For every convex set $A\in\mathcal{OC}^\uparrow$, 
	\begin{equation}\label{eq:JconvexI}
		{J}_A=\underset{n\in\N}\sup \tfrac{1}{n}\log\mathbb{P}(X_n\in A).
	\end{equation}
	Moreover, for every convex set $O\in \mathcal{O}^\uparrow$,  
	\begin{equation}\label{eq:JconvexII}
		{J}_{O}=\lim_{n\to\infty} \tfrac{1}{n}\log\mathbb{P}(X_n\in O)=\underset{n\in\N}\sup \tfrac{1}{n}\log\mathbb{P}(X_n\in O).
	\end{equation}
\end{lemma}
\begin{proof}
	The equalities~\eqref{eq:JconvexI} and \eqref{eq:JconvexII} follow from~\cite[Proposition 12.5]{petit} and \cite[Proposition 12.2]{petit}, respectively.  
    Although~\cite{petit} assumes that the state space is a separable Banach space, these particular results are also valid for general topological vector spaces. 
\end{proof}


\begin{lemma}\label{lem:convexCram}
	$\I\colon E\to [0,\infty]$ is convex. 
\end{lemma}
\begin{proof}
	We adapt the proof of~\cite[Proposition 12.9]{petit}. 
	Let $U,V$ be open and convex sets.   
	For fixed $n\in\N$, define $X_{n+1,2n}=\tfrac{1}{n}\sum_{i=n+1}^{2 n}\xi_i$.  
	Since $X_{2n}=\tfrac{1}{2}(X_n+X_{n+1,2n})$, as well as $X_n$ and $X_{n+1,2 n}$ are independent,  
	\[
	\mathbb{P}(X_{n}\in {\uparrow}U)\mathbb{P}(X_{n+1,2 n}\in {\uparrow}V)=\mathbb{P}(\{X_n\in {\uparrow}U\}\cap\{X_{n+1,2n}\in {\uparrow}V\})\le 
	\mathbb{P}(X_{2n}\in \tfrac{1}{2}{\uparrow}U + \tfrac{1}{2}{\uparrow}V). 
	\] 
	Hence,
	\[ 
	\tfrac{1}{2}\tfrac{1}{n} \log \mathbb{P}(X_{n}\in {\uparrow}U)+  \tfrac{1}{2}\tfrac{1}{n} \log \mathbb{P}(X_{n}\in {\uparrow}V)\le 
	\tfrac{1}{2 n}\log \mathbb{P}(X_{2n}\in \tfrac{1}{2}{\uparrow}U + \tfrac{1}{2}{\uparrow}V).
	\]
	It follows from Lemma~\ref{lem:limExists} that
	\begin{align*}
		{J}_{\tfrac{1}{2}{\uparrow}U + \tfrac{1}{2}{\uparrow}V}=&\lim_{n\to\infty}\tfrac{1}{n}\log\mathbb{P}(X_n\in \tfrac{1}{2}{\uparrow}U + \tfrac{1}{2}{\uparrow}V)\\
		=&\lim_{n\to\infty}\tfrac{1}{2n}\log\mathbb{P}(X_{2n}\in \tfrac{1}{2}{\uparrow}U + \tfrac{1}{2}{\uparrow}V)\\
		\ge & \tfrac{1}{2}\lim_{n\to\infty}\tfrac{1}{n}\log\mathbb{P}(X_n\in {\uparrow}U)+
		\tfrac{1}{2}\lim_{n\to\infty}\tfrac{1}{n}\log\mathbb{P}(X_n\in {\uparrow}V)\\
		= & \tfrac{1}{2}{J}_{{\uparrow}U} + \tfrac{1}{2}{J}_{{\uparrow}V}.
	\end{align*}
	Since the topology of $E$ is generated by the collection of all convex open subsets of $E$,  
	we can apply Lemma~\ref{lem:convexRate} to conclude that $\I$ is convex. 
\end{proof}

\begin{definition}
An $E$-valued random variable $\xi$ is called \emph{convex tight} if for every $\varepsilon>0$ there exists a convex compact $K\subset E$ such that $\mathbb{P}(\xi\in K)\ge 1-\varepsilon$.  Moreover, $\xi$ is said to be \emph{convex inner regular} if for every $\varepsilon>0$ and each convex open $O\subset E$, there exists a convex compact $K\subset O$ such that $\mathbb{P}(\xi \in K)\ge \mathbb{P}(\xi\in O)-\varepsilon$. 
\end{definition}
By adapting~\cite[Proposition 12.7]{petit} to the present setting, 
	we remark that every $E$-valued random variable that is convex tight is also convex inner regular. Moreover, if $E$ is a separable Banach space, then every random variable is automatically convex tight and convex inner regular; see~\cite[Proposition 12.4]{petit} and \cite[Proposition 12.7]{petit}.

From now on, let $\xi$  be an $E$-valued random variable distributed as $\xi_1$.
We need the following result from~\cite{petit}. For the  sake of completeness, we provide a proof.   
\begin{lemma}\label{lem:tightness}
	Let $f\colon E\to (0,\infty)$ be a Borel measurable function, and suppose that $\xi$ is convex tight. 
	Then, for every $\varepsilon>0$, there exists a convex compact $K\subset E$ such that
	\[
		\varepsilon^{-1} \wedge \Big(\log\mathbb{E}_{\mathbb{P}}[f(\xi)]-\varepsilon\Big)\le \log\mathbb{E}_{\mathbb{P}}[f(\xi) 1_K(\xi)].
	\]
\end{lemma}
\begin{proof}
     Suppose first that $f\colon E\to\R$ is bounded.    
	Fix $\varepsilon>0$. 
	Since $\xi$ is convex tight, there exists a compact set $K\subset E$ such that
	\[
	\mathbb{P}(\xi\in K^c)\le (1-e^{-\varepsilon})\frac{\mathbb{E}_{\mathbb{P}}[f(\xi)]}{M},
	\]
	where $|f(x)|\le M$ for all~$x\in E$. 
	Then, 
	\begin{align*}
		\mathbb{E}_{\mathbb{P}}[f(\xi)]&=\mathbb{E}_{\mathbb{P}}[f(\xi)1_{K}(\xi)]+\mathbb{E}_{\mathbb{P}}[f(\xi)1_{K^c}(\xi)]\\
		&\le \mathbb{E}_{\mathbb{P}}[f(\xi)1_{K}(\xi)]+ M\mathbb{P}(\xi \in K^c)\\
		&\le \mathbb{E}_{\mathbb{P}}[f(\xi)1_{K}(\xi)]+ (1-e^{-\varepsilon}){\mathbb{E}_{\mathbb{P}}[f(\xi)]},
	\end{align*}
	which shows that $\log\mathbb{E}_{\mathbb{P}}[f(\xi)]-\varepsilon\le \log\mathbb{E}_{\mathbb{P}}[f(\xi) 1_K(\xi)]$.
	
	In case that $f$ is not bounded, due to the monotone convergence theorem, there exists $N\in\N$ such that 
	\[
	\varepsilon^{-1} \wedge \Big(\log\mathbb{E}_{\mathbb{P}}[f(\xi)]-\varepsilon\Big)\le \log\mathbb{E}_{\mathbb{P}}[f(\xi)\wedge N]-\varepsilon/2.
	\] 
	Since $x\mapsto f(x) \wedge N$ is bounded, it follows from the  first part that there exists a convex compact $K\subset E$ such that 
	\[
	\log\mathbb{E}_{\mathbb{P}}[f(\xi)\wedge N]-\varepsilon/2\le \log\mathbb{E}_{\mathbb{P}}[f(\xi) 1_K(\xi)].\qedhere
	\]
\end{proof}
The \emph{logarithmic moment generating function} of $\xi$ is defined  by 
\[
\Lambda\colon E^\ast_+\to[0,\infty],\quad \Lambda(\mu):=\log\mathbb{E}_{\mathbb{P}}[\exp( \mu(\xi))].
\]
In addition, we define its positive convex conjugate $\Lambda^\ast_+\colon E\to[0,\infty]$ by
$$\Lambda^\ast_+(x)=\underset{\mu\in E^\ast_+}\sup\{\mu(x)-\Lambda(\mu)\}.$$ 
Then the following monotone version of Cram\'{e}r's theorem holds. 
\begin{theorem}\label{thm:CramerRate}
	Suppose that $\xi$ is convex tight.  
	Then,  
	\begin{equation}\label{eq:mgfII}
		\I(x)=\Lambda^\ast_+(x)\quad\mbox{for all }x\in E.  
	\end{equation}
	Moreover,	
	\begin{equation}\label{eq:upperCramer}
		\limsup_{n\to\infty}\tfrac{1}{n}\mathbb{P}(X_n\in C)\le -\inf_{x\in C} \Lambda^\ast_+(x)
		\quad\mbox{ for all }C\in\mathcal{C}_c^\uparrow,
	\end{equation}
	\begin{equation}\label{eq:lowerCramer}
		\liminf_{n\to\infty}\tfrac{1}{n}\mathbb{P}(X_n\in O)\ge -\inf_{x\in O} \Lambda^\ast_+(x)
		\quad\mbox{ for all }O\in\mathcal{O}^\uparrow.	
	\end{equation}
    If additionally $X_n$ is convex tight for all~$n\in\mathbb{N}$, then  
	\[
	\lim_{n\to\infty}\tfrac{1}{n}\mathbb{P}(X_n\in O)
	=-\inf_{x\in O} \Lambda^\ast_+(x)
	\quad\mbox{ for all convex }O\in\mathcal{O}^\uparrow.
	\] 
\end{theorem}
\begin{proof}
	Fix $\mu\in E^\ast_+$.  Since $(\xi_n)_{n\in\N}$ is i.i.d., we have for all $N\in\mathbb{N}$,
	\begin{align*}
		\Lambda(\mu)&=\limsup_{n\to\infty}\tfrac{1}{n}\log\mathbb{E}_{\P}[\exp(n \mu(X_n))]\\
		&\ge \limsup_{n\to\infty}\tfrac{1}{n}\log\mathbb{E}_{\P}[\exp(n(N\wedge \mu)(X_n))]\\
		&=\phi_{J}(N\wedge \mu),
	\end{align*} 
	where the last equality follows from~\eqref{eq:repCramer}.   
	Letting $N\to\infty$, it follows from Lemma~\ref{lem:upperextension} that
		\[
		\Lambda(\mu)\ge \phi_{{J}}(\mu). 
	\]
 Fix $\varepsilon>0$. By Lemma~\ref{lem:tightness}, there exists a convex compact $K\subset E$ such that 
	\[
	\varepsilon^{-1} \wedge \Big(\Lambda(\mu)-\varepsilon\Big)\le \log\mathbb{E}_{\mathbb{P}}[e^{\mu(\xi)} 1_{K}(\xi)].
	\]
	Since $(\xi_n)_{n\in\N}$ is i.i.d. and $K$ is convex,
	\begin{align*}
		\mathbb{E}_{\mathbb{P}}[e^{n \mu(X_n)}1_{K}(X_n)]&\ge \mathbb{E}_{\mathbb{P}}[e^{\mu(\xi_1)}\cdots e^{\mu(\xi_n)} 1_{K}(\xi_1)\cdots 1_{K}(\xi_n)]\\
		&= \mathbb{E}_{\mathbb{P}}[e^{\mu(\xi)}1_{K}(\xi)]^n. 
	\end{align*}
	This shows that 
	\begin{align*}
		\varepsilon^{-1} \wedge \Big(\Lambda(\mu)-\varepsilon\Big)&\le  
		\log\mathbb{E}_{\mathbb{P}}[e^{\mu(\xi)}1_{K}(\xi)]\\
		&\le \limsup_{n\to\infty} \tfrac{1}{n}\log \mathbb{E}_{\mathbb{P}}[e^{n \mu(X_n)}1_{K}(X_n)]\\
		&\le \sup_{x\in E}\big\{\mu(x)-I_{\rm min}(x)\big\}\\
		&=I^\ast_{\rm min}(\mu),
	\end{align*}
where the last inequality follows from Lemma~\ref{lem:boundCramer} in the Appendix B. 
Since $\varepsilon>0$ was arbitrary, we obtain $\phi_J(\mu)\le \Lambda(\mu)\le I^\ast_{\rm min}(\mu)$.
Hence, since the minimal rater function $I_{\rm min}$ is convex, lower semicontinuous and increasing, in line with Proposition~\ref{prop:conI}, we obtain for all $x\in E$,
\begin{align*}
	I_{\rm min}(x)&\ge \sup_{\mu \in E_+^\ast}\{\mu(x)-\phi_J(\mu)\}\\
	&\ge \sup_{\mu \in E_+^\ast}\{\mu(x)-\Lambda(\mu)\}\\
		&\ge \sup_{\mu \in E_+^\ast}\{\mu(x)- I^\ast_{\rm min}(\mu)\}\\
		&=	I_{\rm min}(x).
\end{align*} 
This shows equation \eqref{eq:mgfII}. Moreover, since the concentration $J$ is weakly maxitive,
the upper bound~\eqref{eq:upperCramer} follows directly from Proposition~\ref{prop:LDPupper}.
To show the lower bound \eqref{eq:lowerCramer}, we consider the concentration $\underline{J}_A:=\liminf_{n\to\infty}\tfrac{1}{n}\log \mathbb{P}(X_n\in A)$ with minimal rate function $\underline{I}_{\min}$.  Let $\mathcal{U}$ be a topological base  consisting of open convex sets. 
It follows from Lemma~\ref{lem:ratefunction1} that for every $x\in E$, 
	\[
	\underline{I}_{\min}(x)=-\underset{U\in\mathcal{U}_x}\inf\underline{J}_U=-\underset{U\in\mathcal{U}_x}\inf J_U=\I(x)=\Lambda^\ast_+(x),
	\]
	where $J_U=\underline{J}_U$ is valid due to~\eqref{eq:JconvexII} in Lemma~\ref{lem:limExists}. 
	Then the lower bound \eqref{eq:lowerCramer} follows from the lower bound \eqref{eq:LDPL}, which is satisfied for the minimal rate function $\underline{I}_{\min}=\Lambda^\ast_+$. 

Finally, suppose that for each $n\in\N$ the random variable $X_n$ is convex tight and therefore  convex inner regular. 
Fix $\varepsilon>0$, a convex $O\in\mathcal{O}^\uparrow$ and $N\in\mathbb{N}$.
	Since $X_N$ is convex inner regular, there exists a convex compact $K\subset O$ such that
	\[
	\tfrac{1}{N}\log\mathbb{P}(X_N\in O)\le \tfrac{1}{N}\log\mathbb{P}(X_N\in K) + \varepsilon
	\le \tfrac{1}{N}\log\mathbb{P}(X_N\in {\uparrow}K) + \varepsilon. 
	\] 
	 It follows from Lemma~\ref{lem:compactUp} that ${\uparrow}K\in \mathcal{C}_c^\uparrow$ and ${\uparrow}K$ is a Borel set.
	 Since ${\uparrow}K$ is convex, we can apply Lemma~\ref{lem:limExists} to conclude
	\[
	{J}_{{\uparrow}K}=\limsup_{n\to\infty} \tfrac{1}{n}\log\mathbb{P}(X_n\in {\uparrow}K)
	=\sup _{n\in\mathbb{N}}\tfrac{1}{n}\log\mathbb{P}(X_n\in {\uparrow}K),
	\] 
	and therefore
	\[
	\tfrac{1}{N}\log\mathbb{P}(X_N\in O)\le \tfrac{1}{N}\log\mathbb{P}(X_N\in {\uparrow}K) + \varepsilon \le \underset{n\in\mathbb{N}}\sup \tfrac{1}{n}\log\mathbb{P}(X_n\in {\uparrow}K)  + \varepsilon = {J}_{{\uparrow}K} + \varepsilon. 
	\] 
	By Proposition~\ref{prop:LDPupper}, it holds $J_{{\uparrow}K}\le -\underset{x\in {\uparrow}K}\inf \I(x)\le -\underset{x\in O}\inf \I(x).$ 
	Hence, 
	\[
	\tfrac{1}{N}\log\mathbb{P}(X_N\in O)\le -\underset{x\in O}\inf \I(x) + \varepsilon. 
	\]
	Letting $N\to\infty$ and then $\varepsilon\downarrow 0$, 
	\[
	{J}_{O}=\lim_{N\to\infty}\tfrac{1}{N}\log\mathbb{P}(X_N\in O)\le -\underset{x\in O}\inf \I(x),
	\]
	where the limit above exists due to Proposition \ref{lem:limExists}. 
	The other inequality follows from the lower bound \eqref{eq:LDPL}, which is satisfied for the minimal rate function $\I$.
\end{proof}

The sample mean of an i.i.d.~sequence of random variables with values in~$\R^d$ satisfies the usual large deviation principle with rate function $\Lambda^\ast$; cf.~\cite[Theorem 2.2.30]{dembo}.  In infinite dimensional spaces, the  upper bound in the  large deviation principle   
\begin{equation}\label{eq:usualUpperCr}
\limsup_{n\to\infty}\tfrac{1}{n}\mathbb{P}(X_n\in C)\le -\inf_{x\in C} \Lambda^\ast(x)
\end{equation}
 is only known for certain sets $C\subset E$.    
 For instance, bounds for compact sets or convex open sets are shown in~\cite[Theorem 6.1.3]{dembo},  \cite{fuqing1997note} and \cite{petit}. 
 We obtain in \eqref{eq:upperCramer} a new upper bound for compactly generated sets.   
Moreover, each choice of the cone $E_+$ yields a class of upwards closed sets for which we obtain an upper bound in terms of $\Lambda^\ast_+$.  
By Corollary \ref{cor:tightness}, the upper bound~\eqref{eq:upperCramer} is valid for all $C\in \mathcal{C}^\uparrow$ if the concentration $J_A=\limsup_{n\to\infty}\tfrac{1}{n}\log\mathbb{P}(X_n\in A)$ is tight. 
Moreover, the  lower bound in the large deviation principle   
\begin{equation}\label{eq:usualLowerCr}
\liminf_{n\to\infty}\tfrac{1}{n}\log\mathbb{P}(X_n\in O)\ge -\inf_{x\in O} \Lambda^\ast(x)
\end{equation}
is valid for all open set $O\subset E$; see e.g.~\cite{petit}.  
Since $-\inf_{x\in O}\Lambda^\ast_+\ge -\inf_{x\in O}\Lambda^\ast$, we obtain in \eqref{eq:lowerCramer} a sharper lower bound on upwards closed open sets.

\section{Conclusions}
It is well-known that the maxitive integral with respect to a completely maxitive concentration (capacity) is uniquely determined by a rate function (possibility distribution).    
In this paper, we have introduced the notion of weak maxitivity, and we have shown that, if a concentration $J$ is weakly maxitive and tight, then the corresponding maxitive integral $\phi_J$ 
is determined by the minimal rate function $\I$ on the space of all increasing continuous functions.

Furthermore, every maxitive integral $\phi_J$ is a non-linear expectation with the translation property, 
and we have argued that $\phi_J$ is weakly maxitive if $J$ is weakly maxitive. 
Conversely, we have shown that every weakly maxitive non-linear expection $\psi$ with the translation property has a maxitive integral representation $\psi=\phi_J$ on the space of all increasing continuous  functions. 

Motivated by the theory of large deviations, we have provided different representation results for the minimal rate function $\I$. 
First, under Assumption \ref{ass:sep} (which is satisfied if the state space $E$ is a preordered topological group), we have seen that the rate function is determined by the space of all bounded increasing continuous functions, i.e.,
\[
\I(x)=\underset{f\in C_b^\uparrow}\sup\{f(x)-\phi_J(f)\}.
\]
In addition, we have formulated monotone analogues of the large deviation principle and the Laplace principle, which have been shown to be equivalent under suitable conditions and which uniquely determine the rate function within the class of increasing lower semicontinuous functions.   
Second, we have focused on the case where the rate function is a convex function on a locally convex topological real vector space $E$.    
In that case, under an additional assumption which is implied by the monotone Laplace principle, the minimal rate function is specified by the dual space, i.e.,
\[
\I(x)=\underset{\mu\in E^\ast}\sup\{\mu(x)-\phi_J(\mu)\}.
\]
	Finally, we have shown that standard large deviations theory can be understood within the framework of weakly maxitive concentrations and their maxitive integrals. 
	In turn, the present framework enlarges the scope of large deviations theory to non-standard situations, which we have illustrated with two examples. 
	On the one hand, we have covered the asymptotic concentration of capacities on preordered topological spaces. 
	On the other hand, we have established new large deviation bounds for the sample mean of  i.i.d.~sequences on upwards closed sets by showing a monotone analogue of Cram\'{e}r's theorem on  locally convex topological vector spaces.

 \begin{appendix}
\section{Auxiliary results for preordered topological groups}\label{sec:discussion}
Let $G$ be a Hausdorff topological abelian group. Given a subset $G_+\subset G$ such that 
\[
G_+ + G_+\subset G_+,\quad 0\in G_+,
\]
we endow $G$ with the preorder induced by $G_+$, i.e., $x\le y$ if and only if $y-x\in G_+$. 
Next, we show that $G$ satisfies Assumption~\ref{ass:sep}.

\begin{theorem}\label{thm:separation}
For every $A\in \mathcal{C}^{\uparrow}$ and $x\notin A$ there exists an increasing continuous function $f\colon G\to [0,1]$ such that
\[
f(x)=0\quad\mbox{ and }\quad A\subset f^{-1}(1).
\]
Similarly, for every $A\in \mathcal{C}^{\downarrow}$ and $x\notin A$, there exists an increasing continuous function $f\colon G\to [0,1]$ such that
\[
A\subset f^{-1}(0)\quad\mbox{ and }\quad f(x)=1.
\]
\end{theorem}
\begin{proof}
Notice that it suffices to prove the first statement as the second one can be obtained from the first one by applying the transformation $z\mapsto -z$.  
The argumentation is an adaptation of ~\cite[Theorem 5]{husain2018introduction}, where it is shown  that every topological group is completely regular.  

Suppose that $A\in\mathcal{C}^{\uparrow}$ and $x\notin A$. 
Without loss of generality, we  assume that $x=0$. 
Denote by $\mathcal{D}$ the set of all dyadic rational numbers of the interval $[0,1]$, i.e., $\mathcal{D}:=\{k 2^{-n}\colon n,k\in\mathbb{N},\: k\le 2^n\}$. 
As in the proof of~\cite[Theorem 5]{husain2018introduction}, it is possible to construct a family $(V_r)_{r\in\mathcal{D}}$ of open neighborhoods of $0\in G$ which satisfy the following properties:
\begin{enumerate}
\item $V_1=A^c$,
\item $V_r=-V_r$ for $r<1$,
\item $V_r\subset V_s$ if $r\le s$,
\item $V_{k 2^{-n}}+V_{2^{-n}}\subset V_{(k+1) 2^{-n}}$ for $k=1,\ldots,2^n-1$.\end{enumerate}  
Then, we define
\[
f\colon G\to [0,1],\quad f(x):=1\wedge \inf\{r\in\mathcal{D} \colon x\in {\downarrow}V_r\}.
\]
As usual, we set $\inf\emptyset:=\infty$. 
Since $0\in {\downarrow}V_r$ for all $r$, we have that $f(0)=0$. 
Suppose that $x\in A$. 
Due to (i), (iii), and Lemma~\ref{lem:basictop}, we have that ${\downarrow}V_r\subset A^c$ for all $r$. 
Hence, if $x\in A$, we have that $x\notin {\downarrow}V_r$ for $r$, and consequently $f(x)=1$. 
Finally, since ${\downarrow}V_r$ is downwards closed, it follows that $f$ is increasing. It remains to show that $f$ is continuous. 
To that end, we fix $y\in G$. 

First, suppose that $f(y)=1$. 
Let $\varepsilon>0$ and fix $n\in\N$ with ${2^{-{(n-1)}}}<\varepsilon$. 
We show that $z\in y+ V_{2^{-n}}$ implies that $z\in ({\downarrow}V_{k 2^{-n}})^c$ for all $k<2^n-2$.  
If otherwise $z\in {\downarrow}V_{k{2^{-n}}}$, then due to (ii) and (iv), 
\[
y\in z-V_{2^{-n}}=z+V_{2^{-n}}\subset {\downarrow}V_{{k} 2^{-n}}+V_{2^{-n}}\subset {\downarrow}V_{{k}2^{-n}}+{\downarrow}V_{2^{-n}}\subset {\downarrow}V_{(k+1)2^{-n}},
\]  
and therefore, $f(y)\le(k+1)2^{-n}<(2^n-1)2^{-n}<1$, which is a contradiction to $f(y)=1$. 
Thus, for $z\in x+ V_{2^{-n}}$, we have $f(z)\ge (2^n-2)2^{-n}$, and therefore
\[
|f(y)-f(z)|=1-f(z)\le 1-(2^n-2)2^{-n}=2^{-(n-1)}<\varepsilon.
\]
We conclude that $f$ is continuous at $y$. 

Second, suppose that $0<f(y)<1$. 
Let $\varepsilon>0$ and $n\in\mathbb{N}$ such that $ 2^{-(n-1)}<\varepsilon$, $2^{-(n-1)}<1-f(y)$, and $2^{-(n-1)}<f(y)$.  
Let $k\in\N$ be the smallest $i\in\{1,\ldots,2^n\}$ such that $y\in {\downarrow}V_{i 2^{-n}}$.  
 Then, $(k-1) 2^{-n}\le  f(y)\le  k 2^{-n}$ and
\[
y\in {\downarrow}V_{k 2^{-n}}\cap ({\downarrow}V_{(k-1)2^{-n}})^c.
\] 
Since $1 - f(x)>2^{-(n-1)}$, we have $1-(k-1)2^{-n}\ge 1 - f(x)>2^{-(n-1)}$, so that $k<2^n$. 
Since $f(x)>2^{-(n-1)}$, it holds $k 2^{-n}\ge f(x)> 2^{-(n-1)}$, and therefore $k>2$. 
We have $z\in  ({\downarrow}V_{(k-2)2^{-n}})^c$ whenever $z\in y+V_{2^{-n}}$. 
Otherwise,  due to (ii) and (iv), we obtain
\[
y\in z-V_{2^{-n}}=z+V_{2^{-n}}\subset {\downarrow}V_{(k-2)2^{-n}}+V_{2^{-n}}\subset {\downarrow}V_{(k-2)2^{-n}}+{\downarrow}V_{2^{-n}}\subset {\downarrow}V_{(k-1)2^{-n}},
\]  
 which is a contradiction. 
Moreover, if $z\in y+V_{2^{-m}}$, since $y\in {\downarrow}V_{k 2^{-n}}$, it follows from (iv)    that\footnote{We can apply (iv). Indeed, since $f(x)\ge (k-1)2^{-n}$, we have $1-(k-1)2^{-n}\ge 1 - f(x)>2^{-(n-1)}$, and therefore $k<2^n$. }
\[
z\in  y+V_{2^{-n}}\subset  {\downarrow}V_{k 2^{-n}} + {\downarrow}V_{2^{-n}}\subset {\downarrow}V_{(k+1)2^{-n}}.
\] 
Hence, for $z\in y+V_{2^{-n}}$, we have $(k-2)2^{-n} \le f(z)\le (k+1)2^{-n}$, and therefore
\[
|f(y)-f(z)|\le 2^{-(n-1)}<\varepsilon. 
\]
We conclude that $f$ is continuous at $y$. 

Third, suppose that $f(y)=0$. 
Let $\varepsilon>0$ and fix $n\in\mathbb{N}$ with $2^{-(n-1)}<\varepsilon$. 
Since $f(y)=0$, we have $y\in {\downarrow}V_{2^{-n}}$. 
For $z\in y+V_{2^{-n}}$, it holds
\[
z\in y+V_{2^{-n}}\subset {\downarrow}V_{2^{-n}}+{\downarrow}V_{2^{-n}}\subset 
{\downarrow}V_{2^{-(n-1)}},
\]
and therefore
\[
|f(y)-f(z)|=f(z)\le 2^{-(n-1)}<\varepsilon.
\]
We conclude that $f$ is continuous at $y$. 
The proof is complete.
\end{proof}

\begin{lemma}\label{lem:compactUp}
Suppose that $G_+$ is closed. 
If $K\subset G$ is compact, then ${\uparrow}K$ is closed.
\end{lemma}
\begin{proof}
Suppose that $(x_\alpha)$ is a net in ${\uparrow}K$ such that $x_\alpha\to x$. 
We show that $x\in {\uparrow}K$. 
For each $\alpha$, we have $x_\alpha=y_\alpha + z_\alpha$ where $y_\alpha\in K$ and $z_\alpha\in G_+$. 
Since $K$ is compact there exists a subnet $(y_\beta)$ such that $y_\beta\to y\in K$. 
Then $z_\beta=x_\beta-y_\beta$ converges to $x-y$. 
Since $G_+$ is closed, we have $x-y\in G_+$, and therefore $x\in K+G_+={\uparrow}K$.   
\end{proof}

\section{Auxiliary result for asymptotic concentration of capacities}
Let $(E,\le)$ be a topological preordered space which satisfies Assumption \ref{ass1}. 
In line with Section~\ref{sec:examples}, we consider the concentration $J_A:=\limsup_{n\to\infty}\tfrac{1}{n}\log\mu_n(A)$ with minimal rate function $I_{\rm min}$,
where  $\mu_n(A):=\mathcal{E}_n(1_A)$ for all $A\in  \mathcal{OC}^{\uparrow}$, and  $(\mathcal{E}_n)_{n\in\N}$ is a sequence of sublinear expectations on $\overline{B}$.

\begin{lemma}\label{lem:boundCramer}
Let $K\subset E$ be a compact set. Then, for every $f\in L^\uparrow\cap U^\uparrow$,
\begin{align*}
\limsup_{n\to\infty}\tfrac{1}{n} \log \mathcal{E}_n\big(\exp(nf)1_K\big) &=\phi_J(-\infty1_{K^c}+f1_K)\\
&\le  \sup_{x\in E}\big\{f(x)-I_{\rm min}(x)\big\}.
\end{align*}
\end{lemma}
\begin{proof}
If $J_K=-\infty$, then the first two terms are equal to $-\infty$ as $f$ is bounded on $K$, hence the inequality trivially holds.  
Otherwise, if $J_K>-\infty$, 
we consider the concentration $J^K_A:=J_{A\cap K}-J_K$, $A\in  \mathcal{OC}^{\uparrow}$, and the trivial preorder $\le$\footnote{I.e., $x\le y$ if and only if $x=y$.}. By similar arguments as in Proposition~\ref{eq:seqCap}, for every $f\in \overline{L}^\uparrow\cap \overline{U}^\uparrow$,
\begin{equation}\label{boundCramer2}
	\phi_{J^K}(f)=\phi_J(-\infty1_{K^c}+f1_K)-J_K=\limsup_{n\to\infty}\tfrac{1}{n}\log \mathcal{E}_n(\exp(nf)1_K)-J_K.
\end{equation}
Moreover, since $-\infty1_{K^c}+f1_K\in  U_c^\uparrow$, it follows from Theorem~\ref{thm:mVaradhan1} that
\begin{equation}\label{boundCramer3}
\phi_J(-\infty1_{K^c}+f1_K)\le \sup_{x\in K}\big\{f(x)-I_{\rm min}^0(x)\big\}\le \sup_{x\in E}\big\{f(x)-I_{\rm min}^0(x)\big\},
\end{equation}
where $I^0_{\rm min}$ denotes the minimal rate function for the concentration $J$ w.r.t.~the trivial order.
 By definition of the minimal rate function, it holds $I^0_{\rm min}\ge I_{\rm min}$. Hence, we obtain the claimed assertion as a direct consequence of 
equation~\eqref{boundCramer2} and inequality~\eqref{boundCramer3}.
\end{proof}

\end{appendix}


\begin{thebibliography}{10}

\bibitem{arslanov2004existence}
M.~Z. Arslanov and E.~E. Ismail,
\newblock On the existence of a possibility distribution function,
\newblock {\em Fuzzy Sets Syst.} 148 (2004) 279--290.

\bibitem{artzner1999coherent}
P.~Artzner, F.~Delbaen, J.-M. Eber, and D.~Heath,
\newblock Coherent measures of risk,
\newblock {\em Math. Finance} 9 (1999) 203--228.

\bibitem{bonsall}
F.~F. Bonsall,
\newblock Semi-algebras of continuous functions,
\newblock {\em Proc. London Math. Soc.} 3 (1960) 120--140.

\bibitem{bosi}
G.~Bosi, A.~Caterino, and R.~Ceppitelli,
\newblock Normally preordered spaces and continuous multi-utilities, 
\newblock {\em Appl. Gen. Topol.} 17 (2016) 71--81.

\bibitem{dembo}
A.~Dembo and O.~Zeitouni,
\newblock {\em Large Deviations Techniques and Applications},
\newblock Springer, 2010.

\bibitem{cattaneo2014maxitive}
M.~E. Cattaneo,
\newblock Maxitive integral of real-valued functions, 
\newblock in {\em Information Processing and Management of Uncertainty in Knowledge-Based
Systems}, Part 1, Springer, 2014, pp. 226--23.
 

\bibitem{cattaneo2016maxitive}
M.~E. Cattaneo,
\newblock On maxitive integration,
\newblock {\em Fuzzy Sets Syst.} 304 (2016) 65--81.

\bibitem{chen2016large}
Z.~Chen and X.~Feng, 
\newblock Large deviation for negatively dependent random variables under
  sublinear expectation,
\newblock {\em Comm. Statist. Theory Methods} 
  45 (2016) 400--412.

\bibitem{de2001integration}
G.~de~Cooman,
\newblock Integration and conditioning in numerical possibility theory,
\newblock {\em Ann. Math. Artif. Intell.} 
  32 (2001) 87--123.

\bibitem{cooman1999}
G.~de~Cooman and D.~Aeyels,
\newblock Supremum preserving upper probabilities.
\newblock {\em Inform. Sci.} 118 (1999) 173--212.

\bibitem{dubois2006}
D.~Dubois and H.~Prade, 
\newblock {\em Possibility Theory: An Approach to Computerized Processing of Uncertainty},
\newblock Plenum Press, New York, 1988.

\bibitem{dubois2015possibility}
D.~Dubois and H.~Prade,
\newblock Possibility theory and its applications: Where do we stand?
\newblock in: {\em Springer handbook of computational intelligence}, Springer, 2015, pp 31--60.
  

\bibitem{follmer2016stochastic}
H.~F{\"o}llmer and A.~Schied,
\newblock {\em Stochastic Finance. An Introduction in Discrete Time},
\newblock de Gruyter, 2016.

\bibitem{fuqing1997note}
G.~Fuqing,
\newblock A note on Cramer's theorem,
\newblock in: {\em S{\'e}minaire de Probabilit{\'e}s XXXI}, 
  Springer, 1997, pp. 77--79.

\bibitem{huber}
P.~J. Huber,
\newblock {\em Robust Statistics},
\newblock John Wiley \& Sons, 1981.

\bibitem{husain2018introduction}
T.~Husain,
\newblock {\em Introduction to topological groups},
\newblock W.~B.~Sunders Company, 1966.


\bibitem{kelley}
J.~L. Kelley,
\newblock {\em General Topology},
\newblock Springer-Verlag, 1955.


\bibitem{kolmogorov}
A.~Kolmogorov,
\newblock {\em Foundations of the theory of probability},
\newblock  Chelsea Publishing, 1956.

\bibitem{maslov}
V.~N. Kolokoltsov and V.~P. Maslov,
\newblock {\em Idempotent analysis and its applications}, 
\newblock Springer, 2013.

\bibitem{kupper}
M.~Kupper and J.~M. Zapata,
\newblock Large deviations built on max-stability,
\newblock {\em Bernoulli} 27 (2021) 1001--1027.


\bibitem{minguzzi}
E.~Minguzzi, 
\newblock Normally preordered spaces and utilities,
\newblock {\em Order} 30 (2013) 137--150. 


\bibitem{miranda}
E.~Miranda, G.~de~Cooman, and E.~Quaeghebeur,
\newblock {The Hausdorff moment problem under finite additivity}, 
\newblock {\em J. Theoret. Probab.} 20 (2007) 663--693.


\bibitem{miranda2}
E.~Miranda, 
\newblock {A survey of the theory of coherent lower previsions}, 
\newblock {\em Int. J. Approx. Reason.} 48 (2008) 628--658.

\bibitem{murofushi1993}
T.~Murofushi and M.~Sugeno, 
\newblock Continuous-from-above possibility measures and f-additive fuzzy
  measures: characterization and regularity,
\newblock {\em Fuzzy Sets Syst.} 54 (1993) 351--354. 

\bibitem{nachbin1965topology}
L.~Nachbin,
\newblock {\em Topology and order}, 
\newblock D. Van Nostrand Co., Inc., Princeton, N.J.-Toronto, Ont.-London,
  1965.

\bibitem{peng2019nonlinear}
S.~Peng,
\newblock {\em Nonlinear expectations and stochastic calculus under
  uncertainty: with robust CLT and G-Brownian motion}, 
\newblock Springer Nature, 2019.

\bibitem{petit}
P.~Petit,
\newblock Cram{\'e}r’s theorem in banach spaces revisited, 
\newblock in: {\em S{\'e}minaire de Probabilit{\'e}s XLIX}, 
  Springer, 2018, pp. 455--473.

\bibitem{puhalskii2001large}
A.~Puhalskii,
\newblock {\em Large deviations and idempotent probability}, 
\newblock Chapman and Hall/CRC, 2001.

\bibitem{semadeni1968preordered}
Z.~Semadeni and H.~Zidenberg, 
\newblock On preordered topological spaces and increasing semicontinuous
  functions, 
\newblock {\em Prace Mat.} 11 (1968) 313--316.

\bibitem{shilkret1971maxitive}
N.~Shilkret,
\newblock Maxitive measure and integration,
\newblock {\em Indag. Math.} 33 (1971) 109--116.

\bibitem{tan2020large}
Y.~Tan and G.~Zong, 
\newblock Large deviation principle for random variables under sublinear
  expectations on $\R^d$.
\newblock {\em J. Math. Anal. Appl.} 488 (2020) 124110.

\bibitem{walley1991statistical}
P.~Walley, 
\newblock {\em Statistical reasoning with imprecise probabilities}, 
\newblock Chapman and Hall, London, 1991.

\bibitem{zadeh1978fuzzy}
L.~A. Zadeh, 
\newblock Fuzzy sets as a basis for a theory of possibility,
\newblock {\em Fuzzy Sets Syst.} 1 (1978) 3--28.

\end{thebibliography}

\end{document}